\documentclass[a4paper,12pt]{article}

\usepackage{amsthm}
\usepackage{amsfonts}
\usepackage{amsmath}
\usepackage{appendix}
\usepackage{tikz}
\usepackage[normalem]{ulem}
\usepackage{xspace}
\usetikzlibrary{shapes,arrows}

\newtheorem{theorem}{Theorem}[section]
\newtheorem{lemma}[theorem]{Lemma}

\newtheorem{corollary}[theorem]{Corollary}
\newtheorem{definition}{Definition}
\newtheorem{conjecture}{Conjecture}

\newcommand{\gnp}{\ensuremath{G(n,p)}}

\newcommand{\E}{\ensuremath{\mathbb{E}}}
\newcommand{\supGr}{\ensuremath{\mathcal{H}}}
\newcommand{\supGNP}{\ensuremath{\supGr(n,p)}}
\newcommand{\supFact}{\ensuremath{\supGr\mbox{-Factor}}}
\newcommand{\nV}{\ensuremath{v_H}}
\newcommand{\nE}{\ensuremath{e_H}}
\newcommand{\nVH}{\ensuremath{k_{\mathcal{H}}}}
\newcommand{\nEH}{\ensuremath{h_{\mathcal{H}}}}
\newcommand{\entropyFunction}{\ensuremath{h}}
\newcommand{\sumOverPhiPolynomialNotH}{\ensuremath{\tau}}
\newcommand{\wxy}{\ensuremath{w_{x,y}}}
\newcommand{\wyx}{\ensuremath{w_{y,x}}}
\newcommand{\weight}{\ensuremath{w}}

\newcommand{\nVmult}{\ensuremath{v_\supGr}}
\newcommand{\set}[1]{ \left\{ #1 \right\} }
\newcommand{\BB}{B}
\newcommand{\ignore}[1]{\relax}
\renewcommand{\th}{\operatorname{th}}
\newcommand{\Order}{\mathcal{O}}
\newcommand{\Rucinski}{Ruci{\'n}ski\xspace}

\newcommand{\genericScale}{1}
\newcommand{\genericThickness}{0.75mm}
\newcommand{\genericCirclePt}{4pt}

\newcommand{\comments}[1]{}

\begin{document}
\bibliographystyle{amsplain}

\title{Non-Vertex-Balanced Factors in Random Graphs}
\author{Stefanie Gerke%
\thanks{%
Mathematics Department, 
Royal Holloway College, 
University of London,
Egham TW20 0EX, England.}
\and
Andrew McDowell$^*$%
\thanks{%
\hbox{e-mail}~{\small\texttt{Andrew.McDowell.2010@live.rhul.ac.uk}}}
}
\date{\today}
\maketitle

\begin{abstract}
We prove part of a conjecture by Johansson, Kahn and Vu \cite{JKV} regarding threshold functions 
for the existence of an $H$-factor in a random graph \gnp.
We prove that the conjectured threshold function is correct for any graph $H$ 
which is not covered by its densest subgraphs.
% that contain at least one vertex that lies only in subgraphs less dense than the densest subgraph of $H$. 
We also demonstrate that the main result of \cite{JKV} 
% can be applied to fixed vertex partitions of \gnp, and 
generalises to multigraphs, digraphs,
and a multipartite model.

\end{abstract}


\section{Introduction}

We will properly state our theorems later in the introduction,
after introducing necessary notation and background.
However, for readers already familiar with the background or
willing to momentarily gloss over the details,
let us immediately sketch our main results and methods.
In a recent breakthrough (winning a 2012 Fulkerson Prize),
Johansson, Kahn, and Vu \cite{JKV}
determined the threshold for a random graph $G$
to be factorable by a strictly balanced fixed graph $H$,
and they conjectured the threshold for every $H$.
Our main result, Theorem \ref{MainTheoremIntro},
establishes their conjecture for `non-vertex-balanced' graphs $H$,
a class of graphs disjoint from strictly balanced ones.
The positive side is the difficult one,
and the main idea of the proof is, to cover a fraction of $G$
with copies of a densest subgraph of $H$,
then contract that subgraph to a point, extend the cover, and repeat.
However, after the first step, two things have changed:
the graph $H$ may have become a multigraph and, more significantly,
we have committed to correspondences between some vertices of $G$ and $H$.
We manage these difficulties through
Theorem \ref{MainTheoremPartionMultigraphResult},
asserting that, if the vertices of the random graph
$G$ are partitioned into classes corresponding to vertices of $H$,
then $G$ almost surely has an $H$-factor which respects the partitioning.
Our proof follows the steps of the proof in \cite{JKV};
it is not especially inventive, but neither is it easy.
Krivelevich \cite{2010arXiv1007.2326K_Spanning_trees_Krivelevich}
needed a special case and verified it, but
\cite{2010arXiv1007.2326K_Spanning_trees_Krivelevich}
does not include the proof.
The result is clearly useful,
and it is therefore worth writing down the proof details.
We prove a similar result for directed graphs.

Formally, for graphs $H$ and $G$, an \emph{$H$-factor} of $G$ is a collection of vertex-disjoint copies of $H$ in $G$ that  form a vertex cover of $G$.  
Clearly $G$ can only contain an $H$-factor if $|V(H)|$ divides $|V(G)|$. We are mainly interested in random graphs on $n$ vertices and 
and we assume throughout the paper that $|V(H)|$ divides $n$.

We call a function $f(n)$ a \emph{threshold} for a graph property $K$ if,
for an Erd\H{o}s-R\'{e}nyi random graph \gnp, (that is, the graph on $n$ vertices where each edge is present with probability $p$ independently of the absence or presence of any other edge)
\[
% \begin{array}{ll} 
\Pr(\gnp \mbox{ satisfies }K)\rightarrow 
\begin{cases}
1 & \text{if } p(n)=\omega(f(n)), \text{ and}  \\
0 & \text{if } p(n)=o(f(n)) .
\end{cases}
% 1 \mbox{ if } p(n)=\omega(f(n)), \mbox{ and}  \\
% \Pr(\gnp \mbox{ satisfies }K)\rightarrow 0 \mbox{ if }p(n)=o(f(n)).
% \end{array}
\]

Since containing an $H$-factor is an increasing property (that is, adding edges does not destroy any $H$-factor) it is well known that a threshold function exists, see for example \cite{0471175412_Janson_Luczak_Rucinski_RandomGraphs}. 
Note that a threshold is unique up to  multiplicative positive constants so we will use $\Theta$ notation  and with slight abuse of language we will speak of ``the" threshold.
 The study of thresholds % in \gnp\ 
for various classes of graphs $H$
has attracted considerable interest.
The distinctions center around density properties of $H$.
We define the \emph{density}  of a graph $H$ on at least two vertices, as
$$d(H)=\frac{|E(H)|}{|V(H)|-1} . $$
Let $m(H)$ be the maximum density of any subgraph of $H$, that is,
$$m(H) = \max \set{d(H') \colon H' \subseteq H, |H'| \geq 2} . $$
A graph $H$ is called \emph{balanced} if $m(H) = d(H)$,
i.e., if no subgraph of $H$ has density greater than that of $H$,
and \emph{strictly balanced} if every proper subgraph of $H$ 
has density smaller than that of $H$.

For any vertex $v$ of $H$,
define the \emph{local density} at $v$ to be the maximum density restricted
to subgraphs containing $v$,
% $m(v,H) = \max_{H' \subseteq H \colon v \in V(H'), |H'| \geq 2}\set{d(H')}$.
$$m(v,H) = \max \set{d(H') \colon H' \subseteq H, |H'| \geq 2, v \in V(H')} . $$
A graph $H$ is \emph{vertex balanced} if, for all $v \in H$, $m(v,H)=m(H)$.

Note that if $H$ is balanced then it is vertex balanced:
a densest subgraph of $H$ is $H$ itself, 
so $m(v,H)$ and $m(H)$ are both given by $H'=H$, for $m(v,H)=m(H)=d(H)$.
Taking the contrapositive,
if $H$ is non-vertex balanced then it is not balanced,
and not strictly balanced.
Graphs may thus be partitioned into 
those that are non-vertex balanced, 
those that are strictly balanced, 
and the rest 
(those that are vertex balanced but not strictly balanced).
An example of a non-vertex balanced graph is shown below.
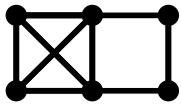
\begin{center}
\pgfdeclarelayer{background}
\pgfsetlayers{background,main}
\begin{tikzpicture}[scale=\genericScale]

		\path (1,1) node (K1) {};
		\path (1,2) node (K2) {};
		\path (2,1) node (K3) {};
		\path (2,2) node (K4) {};

		\path (3,1) node (T1) {};
		\path (3,2) node (T2) {};

		\fill (K1.east) circle (\genericCirclePt);
		\fill (K2.east) circle (\genericCirclePt);
		\fill (K3.east) circle (\genericCirclePt);
		\fill (K4.east) circle (\genericCirclePt);

		\fill (T1.east) circle (\genericCirclePt);
		\fill (T2.east) circle (\genericCirclePt);

		\foreach \j in {1,2,3,4}
		{
			\foreach \i in {1,2,3,4}
			\draw[line width=\genericThickness] (K\j.east)--(K\i.east);
		}

		\foreach \j in {1,2}
		{
			\foreach \i in {1,2}
			\draw[line width=\genericThickness] (T\j.east)--(T\i.east);
		}

		\draw[line width=\genericThickness] (T2.east)--(K2.east);
		\draw[line width=\genericThickness] (T1.east)--(K3.east);

	\end{tikzpicture}

\label{fig:K4K3}
\end{center}

The thresholds for $H$-factors for various fixed graph $H$ have been of interest for a long time.  
The case $H=K_2$ is simply the threshold for $G$
to have a perfect matching which has been known since 1966 \cite{er66}, see also  \cite{BT85} for a more precise result. 
The next $H$-factor threshold result was for trees by {\L}uczak and Rucinski
\cite{DBLP:journals/siamdm/LuczakR91_Rucinski_TreeFactors}. 
Note that matchings and trees are vertex-balanced. 
For sub-classes of non-vertex-balanced graphs, the threshold is known  for graphs $H$ whose minimum degree is less than $m(H)$ 
\cite{
CambridgeJournals:1771816_Alon_Noga_yuster_Raphael_Factors, 0471175412_Janson_Luczak_Rucinski_RandomGraphs}. 
In 2008 
the seminal paper by Johansson, Kahn, and Vu \cite{JKV} 
determined the threshold for all strictly balanced graphs
(also resolving the so-called `Shamir's problem' on hypergraph matchings).
The special case of finding the threshold of an $H$-factor  for the strictly balanced graph  $H=K_3$ had been described by Janson, {\L}uczak and \Rucinski as one of the two
`most challenging, unsolved problems in the theory of random structures'
\cite[p.~96]{0471175412_Janson_Luczak_Rucinski_RandomGraphs}  
and was first posed by  \Rucinski in 1992
\cite{Andrzej1992185_MathcingCopiesOfAGraph_ByDistinctCopies}
(the second problem was 'Shamir's problem').
\ignore{
Two of
the most challenging, unsolved problems in the theory of random structures
are finding the thresholds for the existence of a K3-factor in the random
graph G(n,p) and for the existence of a perfect matching (a collection of n/3
disjoint triples) in a random 3-uniform hypergraph. The latter, known as
the Shamir problem, goes back to at least Schmidt and Shamir (1983). The
former, discussed in greater generality in the previous section, was probably
first stated in Rucinski (1992b).
Some believe that these two problems are immanently related and a solution
of one of them will yield a solution of the other one. Let us point out,
however, that in the hypergraph case the triples are (in the binomial model)
independent from each other, while the triangles of G(n,p) are not. What
certainly links is that, after a quiet period, recently significant
progress with respect to both of them. In this section an
account on this progress is given.
}

In their  paper Johansson, Kahn and Vu conjecture thresholds for all graphs $H$, depending 
on whether $H$ is vertex-balanced or not. 
We restate this formally as Conjecture \ref{MainConjecture} 
in Section \ref{Chapter:NewResults}. 
Our first main result
% Theorem \ref{MainTheoremIntro} 
establishes this conjecture 
for all graphs in the second category. More precisely, 
let $\th_H(n)$ be the threshold function for \gnp\ to contain an $H$-factor.  We prove the following.

\begin{theorem}\label{MainTheoremIntro}
If $H$ is non-vertex-balanced,
\[\th_H(n)=\Theta\left(n^{-1/m(H)}\right).\]
\end{theorem}

% Recalling that a non-vertex balanced graph is not strictly balanced,
% the threshold for such a graph is never given by \cite{JKV}.

% While strictly balanced graphs, for which this conjecture has been proved, are vertex-balanced, most of the previous factor results that had been established, mentioned above, prior to \cite{JKV}, were non-vertex-balanced, but we are now able to prove it for all such graphs. 

The main idea of the proof is, first,
to  embed the dense subgraphs of $H$, 
giving a `partial factor' covering
a corresponding proportion of the vertices of \gnp. 
We then collapse each such subgraph of $H$ to a single vertex,
giving a less dense strictly balanced graph (or possibly multigraph).
Finally, we extend the partial factor to a full factor
using Theorem~\ref{MainTheoremPartionMultigraphResult}, 
a generalisation to partitioned multigraphs
of the strictly balanced result of \cite{JKV}.

To state  Theorem~\ref{MainTheoremPartionMultigraphResult} 
we need some more  notation. 
Let $\nE=|E(H)|$, $\nV=|V(H)|$ and  $V(H)=\{x_1,x_2,\dots, x_{\nV}\}$. 
Define the $r$-fold blowup $\BB(H,r)$ of $H$
as an \nV-partite graph with parts $V_1, V_2,\dots, V_{\nV}$, each of size $r$, 
with an edge between $v_i\in V_i$ and $v_j\in V_j$ 
iff there is an edge between $x_i$ and $x_j$ in $H$. 
In a slight abuse of notation,
let $H(n,p)$ be the random subgraph of $\BB(H,n/\nV)$
obtained by retaining each edge with probability $p$.
% Define $H(n,p)$ to be a random \nV-partite graph, 
% with parts $V_1, V_2,\dots, V_{\nV}$ each of size $n/\nV$, 
% and an edge between vertices $v_i\in V_i$ and $v_j\in V_j$ with probability $p$ if there exists an edge in $H$ between $x_i$ and $x_j$ and 0 otherwise. 
% This can be thought as a random blow-up of $H$. 
Likewise, given a multigraph \supGr, 
we define the random multigraph \supGNP: 
the blowup $\BB(\supGr, n/\nVmult)$ has as many edges between 
$v_i \in V_i$ and $v_j \in V_j$ 
as there are edges between $x_i$ and $x_j$ in \supGr,
and again $\supGNP$ is the random subgraph of $\BB(\supGr,n/\nVmult)$
obtained by retaining each edge with probability $p$.
% as above, except that, where we have vertices in \supGr\ with multiple edges between them, we have an equal number of possible edges, each independently present with probability $p$, between the vertices of the corresponding partitions. Definitions of density are identical to those for multigraphs. 

The setup suggests looking for a `restricted' \supGr-factor of \supGNP\
where, for each copy of \supGr,
each vertex belongs to the corresponding part of \supGNP.
We will not in fact assume this, but the following theorem shows that
below some threshold there is no factor,
while above the threshold there is a factor of the restricted form.

\begin{theorem}\label{MainTheoremPartionMultigraphResult}
 \label{MainTheoremPartionResult}
Fix a multigraph \supGr
(which may be a simple graph $H$).
If \supGr\ is strictly balanced, 
then the threshold % function $\th_{\supGr}(n)$, 
for % the random $|V(\supGr)|$-partite multigraph 
\supGNP\ to contain an \supGr-factor is
\[\th_{\supGr}(n)=\Theta\left( n^{-1/m(\supGr)}(\log n)^{1/|E(\supGr)|} \right),\]
while if \supGr\ is not strictly balanced, the threshold satisfies
\[\th_{\supGr}(n)=\Order(n^{-1/m(\supGr)+o(1)}).\]
In both cases, above the threshold there is w.h.p.\
an \supGr-factor in which, in every copy of \supGr, 
each vertex is in the corresponding part of \supGNP.
\end{theorem}

We note that there is a key difference between the partitioned and usual \gnp\ thresholds for non-strictly balanced graphs. In \gnp, we show that for non-vertex balanced $H$, the $o(1)$ term can be completely eliminated, while it remains in the form of a $\log$ term for  strictly balanced $H$. In the partitioned random graph however, there will always be a $\log$ term. This can be seen by considering the graph consisting of a triangle and a single isolated vertex.

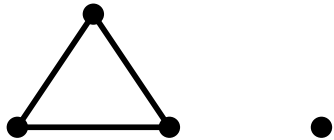
\begin{center}
\pgfdeclarelayer{background}
\pgfsetlayers{background,main}
\begin{tikzpicture}[scale=\genericScale]

	%draw nodes from X20 to X41

		\foreach \j in {1,3,5}
		{
			\path (\j,1) node (X\j1) {};
			\fill (X\j1.east) circle (\genericCirclePt);
		}
		\foreach \j in {2}
		{
			\path (\j,2.5) node (X\j3) {};
			\fill (X\j3.east) circle (\genericCirclePt);
		}

			\draw[line width=\genericThickness ] (X11.east)--(X31.east);
			\draw[line width=\genericThickness] (X23.east)--(X31.east);
			\draw[line width=\genericThickness] (X11.east)--(X23.east);

			\typeout{end}

\end{tikzpicture}

\end{center}

In \gnp, this graph is easy to embed, as it is equivalent to a partial factor of triangles, taken over the whole graph, but only covering $3/4$ of the vertices. At this point, the remaining spare vertices immediately complete the factor. In the partitioned case, by fixing the position of these triangles, we are  implying the existence of a full triangle factor over those corresponding partitions, and as such, by the result for strictly balanced graphs, we will require a $\log$ term, corresponding to the densest subgraphs.

\comments{
Note that we cannot eliminate the $o(1)$ term 
from the threshold for non-strictly balanced graphs 
as there are values of $p$  satisfying $p=\Order(n^{-1/m(H)})$, \marginpar{what does this mean? For any constant C?} such that, 
with high probability, many vertices of \gnp\ will not be covered by copies 
of the densest subgraphs of $H$, 
and hence we will not be able to assign all vertices these `roles' 
in the factor at will. \marginpar{reference? Or is it obvious?}
}

\comments{
Theorem \ref{MainTheoremPartionMultigraphResult} is useful in its own right.
Given a fixed $\nV$-part equipartitioning of the vertices of $G(n,p)$,
Krivelevich \cite{2010arXiv1007.2326K_Spanning_trees_Krivelevich}
finds the threshold for a cycle factor in which each copy of the cycle 
is required to have one vertex in each part.
This is a special case of
Theorem \ref{MainTheoremPartionMultigraphResult}.
}

Lastly we prove that the threshold for digraph factors 
coincides with that for graphs which is an easy consequence of Theorem~\ref{MainTheoremPartionMultigraphResult}.
We define the random directed graph $D(n,p)$ with vertex set $V$ of size $n$, 
such that for each pair of vertices $u$ and $v$ in $V$, 
there is an edge between them with probability $p$ 
independently of all other edges, 
and each such edge is either $(u,v)$ or $(v,u)$, 
with probability half each. 
We can prove the threshold for both strictly balanced 
and non-vertex-balanced digraphs, in $D(n,p)$. 
Note that for strictly balanced graphs we can also prove the partitioned form, 
i.e., the digraph form of Theorem \ref{MainTheoremPartionResult} also holds.

\begin{theorem}\label{MainTheoremDigraphResult}
Fix a digraph $H$. If $H$ is strictly balanced, 
then the threshold function $\th_{H}(n)$, for the random directed graph $D(n,p)$ to contain an $H$-factor is
\[\th_{H}(n)=\Theta\left( n^{-1/m(H)}(\log n)^{1/\nE} \right),\]
	while if $H$ is non-vertex-balanced,
\[\th_{H}(n)=\Theta\left(n^{-1/m(H)} \right).\]
\end{theorem}

\section{Preliminaries}\label{Chapter:NewResults}
We first note the following known results which we will need later.

\begin{theorem}\label{partial_factors}
(\Rucinski \cite{Andrzej1992185_MathcingCopiesOfAGraph_ByDistinctCopies})
Let $H$ be a graph with at least one edge, and $F_H(\varepsilon,n)$ be threshold function for the property that \gnp\ contains a partial $H$-factor covering all but at most $\varepsilon n$ vertices. Then for any fixed $\varepsilon>0$ the threshold function satisfies,
\[F_H(\varepsilon,n)=\Theta\left(n^{-1/m(H)}\right).\]

\end{theorem}

\begin{theorem}\label{mindegree_factors}
(Alon, Yuster \cite{CambridgeJournals:1771816_Alon_Noga_yuster_Raphael_Factors})
Let $H$ be a graph with minimum degree $\delta(H)$, satisfying $\delta(H) < m(H)$. Then 
\[th_H(n)=\Theta\left( n^{-1/m(H)}\right).\]

\end{theorem}

In their respective papers, stronger results than what is stated above are actually proved, but we use threshold notation for consistency. 

The following two results can be found as Theorem 2.1 and 2.2 in \cite{JKV}. This paper's aim is to provide a generalisation of the first result, which allows for improved bounds on the second.
\begin{theorem}  \cite{JKV} \label{strictbal_factors}
Let $H$ be a strictly balanced graph with \nE\ edges. Then the threshold function, $\th_H(n)$ for \gnp\ to contain an $H$-factor satisfies
\[\th_H(n)=\Theta(n^{-1/d(H)}(\log n)^{1/\nE}).
\]

\end{theorem}

\begin{theorem}  \cite{JKV} \label{upperboundgeneral_factors}
For $H$, an arbitrary fixed graph, the threshold function, $\th_H(n)$ for \gnp\ to contain a $H$-factor satisfies
\[\th_H(n)=\Order(n^{-1/m(H)+o(1)}).
\]

\end{theorem}

In \cite{JKV}, the authors define threshold functions $\th_H^{[1]}(n)$ and $th^{[2]}_H(n)$, for a given fixed graph $H$. Firstly, $\th_H^{[1]}(n)$ is defined as the threshold for every vertex in $G$ to be covered by at least one copy of $H$, while $\th_H^{[2]}(n)$ is the threshold for the property of satisfying the following two conditions:
\begin{enumerate}
\item{every vertex of $G$ is covered by at least one copy of $H$, and}
\item{for each $x\in V(H)$, there are at least $n/\nV$ vertices $x'$ of $G$ for which some isomorphism of $H$ takes $x$ to $x'$.}
\end{enumerate}
This threshold is clearly a lower bound for the threshold for finding a factor $\th_H(n)$ and the paper conjectures that they are, in fact, equal, and they prove this for strictly balanced $H$. In this paper, we will show this conjecture also holds for a large class of non-strictly balanced graphs.

The threshold $th^{[2]}_H(n)$ is completely determined and stated without proof in \cite{JKV}, for all graphs. For completeness we have included a proof below. Let $s_v = \min\{e(H'): H'\subseteq H, v\in V(H'), d(H')=m(v,H)\}$ and let $s$ be the maximum over all $s_v$. Clearly $m(v,H)\leq m(G)$ for all $v$, with equality for at least one $v$.
We are now ready to state and prove the following:

\begin{lemma}\label{2ConditionCoveringLemma}
If for all $v\in V(H)$, $m(v,H)=m(H)\ \mbox{ then }$
\[\th_H^{[2]}(n)=\Theta \left( n^{-1/m(H)}(\log(n))^{1/s}\right).\]
Otherwise
\[\th_H^{[2]}(n)=\Theta \left( n^{-1/m(H)} \right).\]
\end{lemma}
\begin{proof}
We clearly have 4 cases to consider. Firstly we will look at vertex-balanced graphs, i.e. those that satisfy, $m(v,H)=m(H)$ for all $v\in V(H)$. Condition 1 of $\th_H^{[2]}(n)$, namely that each vertex of $G$ is covered by at least one copy of $H$, is well studied and exact thresholds can be found as Theorem 3.22 in \cite{0471175412_Janson_Luczak_Rucinski_RandomGraphs} and follow from results proved by Spencer in \cite{DBLP:journals/jct/Spencer90a:Extension_Thresholds}, and, in this case they are equal to our required bound.

Now we simply have to prove that condition 2 is also satisfied for  $p=\omega(n^{-1/m(H)}(\log(n))^{1/s})$. We note that we clearly have that $Cp=\omega(n^{-1/m(H)})$ for any constant $C$. 

Let $V(H)=\{1,2,\dots, \nV\}$, our result will follow from partitioning the edge set of \gnp\ into the union of random graphs $G_0,G_1,\dots,G_{\nV}$ with edge probability $p'$, where $1-p=\prod_{i=0}^{\nV} (1-p')=(1-p')^{\nV+1}$ and repeatedly applying Theorem \ref{partial_factors} to find partial factors of $H$. We first apply it in $G_0$, which has edge probability $p'>p/(\nV+1)=\omega\left(n^{-1/m(H)}\right)$, which is sufficient to apply \ref{partial_factors} with $\varepsilon=1/4$. This gives us a partial $H$-factor covering $3n/4$ of the vertices of \gnp\ and hence $(1-\varepsilon)n/\nV=3n/(4\nV)$ vertices of \gnp\ are covered by each vertex of $H$ with high probability. 

For each $i=\{1,2,\dots, \nV\}$ we consider vertex $i$ of $H$ and the vertices of $G_0$, that we have already covered by copies of $i$, and then, the random graph induced by the edges of $G_i$ on the vertex set of $G_0$, without those already covered vertices. This leaves us with a set of $n'=(1-3/(4\nV))n$ vertices in each $G_i$, that have not already been covered by a copy of the vertex $i$ of $H$, with an independent random edge set. We can consider this as equivalent to the random graph $G(n',p')$, where $p'>p/(\nV+1)=\omega\left((n')^{-1/m(H)}\right)$ (assuming $|\nV |>2$). This allows us to again, apply Theorem \ref{partial_factors} to find another set of partial factors on $3/4$'s of the remaining vertices, giving us in total $(6/4\nV -9/16\nV^2)n>n/\nV$, for $\nV>1$, vertices covered by vertex $i$ of $H$ as required. 

We now consider graphs that are non-vertex-balanced and so do not satisfy $m(v,H)=m(H)$ for all $v\in V(H)$. As before, the threshold for covering is known, and is in fact lower than our required threshold here. 

The same argument for proving condition 2 as above applies since we only required $p=\omega(n^{-1/m(H)})$, so it follows that, for these graphs, both conditions are satisfied for $p=\omega(n^{-1/m(H)})$. It only remains to show that the threshold is not lower than this for such $H$. This follows from another result, proved by \Rucinski and Vince \cite{DBLP:journals/combinatorica/RucinskiV88}. They prove, that for any vertex of \gnp, the threshold for it being covered by a particular vertex $v$ of $H$ is $n^{-1/m(v,H)}$. 

With the result above in mind, we define the following; for a vertex  $v_G\in V(\gnp)$, we let $X_{v_G}$ be the indicator variable for $v_G$ being covered by a copy of $v$, where $v\in V(H)$ satisfies $m(v,H)=m(H)$, namely $X_{v_G}=0$ if it is not covered, and $X_{v_G}=1$ if it is. Suppose that condition 2 is satisfied with high probability. Therefore we have that
\[\mathbb{E}(\sum_{v_G\in \gnp} X_{v_G})=\sum_{v_G\in \gnp}\mathbb{E}( X_{v_G})>n/\nV.\]

Suppose that $p=o(n^{-1/m(v,H)})$.
% (i.e. $\th_H^{[2]}(n)<<n^{-1/m(H)}$). 
We know that for $p$ in this range, $X_{v_G}=0$, with high probability, and therefore $\mathbb{E}( X_{v_G})=o(1)$. Since there are only $n$ such choices of $v_G$, we have a contradiction. Therefore the threshold $\th_H^{[2]}(n)$ is not $o(n^{-1/m(H)})$, and so must be $n^{-1/m(H)}$, as required. 
\end{proof}

In \cite{JKV}, it is conjectured that $\th_{H}^{[2]}(n)=\th_H(n)$, so in light of the above, this can be restated as the following

\begin{conjecture}\label{MainConjecture}\cite{JKV}
% (Johansson, Kahn, Vu, 2008)
If for all $v\in V(H)$, $m(v,H)=m(H), \mbox{ then }$
\[\th_H(n)=n^{-1/m(H)}(\log(n))^{1/s}.\]
Otherwise
\[\th_H(n)=n^{-1/m(H)}.\]
\end{conjecture} 

\section{Theorem \ref{MainTheoremIntro}}\label{sectionOutline}
The first case of Conjecture~\ref{MainConjecture} has been proved for strictly balanced $H$, and now we will prove the second statement in its entirety, namely we prove that the threshold for containing an $H$-factor is $\th_H(n)=\th_H^{[2]}(n)=n^{-1/m(H)}$ for graphs where $m(v,H)<m(H)$ for some $v\in V(H)$. We begin by demonstrating that the result follows, assuming Theorem \ref{MainTheoremPartionMultigraphResult}, and then in Section \ref{Section:proofOfMultiGraphResult}, we return to prove it.

In general terms, the main idea of this paper, is to `collapse' dense sub-graphs of $H$ to get a new graph (or possibly multigraph) \supGr, which we will formally define later. 

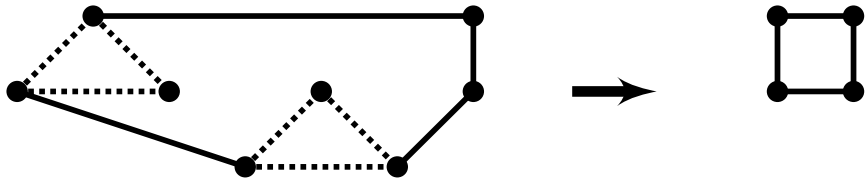
\newcommand{\CThickness}{0.75mm}
\newcommand{\CCirclePt}{4pt}
\newcommand{\CScale}{1}
%\draw[line width=\BThickness]

\begin{center}
\pgfdeclarelayer{background}
\pgfsetlayers{background,main}
\begin{tikzpicture}[scale=\CScale]
	%Creates the top vertex at {x,y} called node {} with label {}
	%\path (5,-2) node (X1) {};
	%\fill (X1.east) circle (1pt);

	%draw nodes from X20 to X41

		\foreach \j in {2,7}
		{
			\path (\j,3) node (X\j3) {};
			\fill (X\j3.east) circle (\CCirclePt);
		}

		\foreach \j in {1,3,5,7}
		{
			\path (\j,2) node (X\j2) {};
			\fill (X\j2.east) circle (\CCirclePt);
		}

		\foreach \j in {4,6}
		{
			\path (\j,1) node (X\j1) {};
			\fill (X\j1.east) circle (\CCirclePt);
		}
		\begin{pgfonlayer}{background}
		%triangle 1
		\draw[line width=\CThickness][dotted] (X12.east)--(X32.east);
		\draw[line width=\CThickness][dotted] (X23.east)--(X32.east);
		\draw[line width=\CThickness][dotted] (X12.east)--(X23.east);
		%triangle 2
		\draw[line width=\CThickness][dotted] (X41.east)--(X61.east);
		\draw[line width=\CThickness][dotted] (X41.east)--(X52.east);
		\draw[line width=\CThickness][dotted] (X52.east)--(X61.east);
		\end{pgfonlayer}

		%Invisible nodes
		\path (8.3,2) node (hidden1) {};
		\path (9.7,2) node (hidden2) {};
		\path [draw, -latex', line width=4]  (hidden1) -- (hidden2);

		\draw[line width=\CThickness] (X12.east)--(X41.east);
		\draw[line width=\CThickness] (X23.east)--(X73.east);
		\draw[line width=\CThickness] (X61.east)--(X72.east);
		\draw[line width=\CThickness] (X73.east)--(X72.east);

		\path (11,2) node (Y12) {};
		\fill (Y12.east) circle (\CCirclePt);
		\path (11,3) node (Y13) {};
		\fill (Y13.east) circle (\CCirclePt);
		\path (12,2) node (Y22) {};
		\fill (Y22.east) circle (\CCirclePt);
		\path (12,3) node (Y23) {};
		\fill (Y23.east) circle (\CCirclePt);

		\draw[line width=\CThickness] (Y12.east)--(Y13.east);
		\draw[line width=\CThickness] (Y13.east)--(Y23.east);
		\draw[line width=\CThickness] (Y23.east)--(Y22.east);
		\draw[line width=\CThickness] (Y22.east)--(Y12.east);

\end{tikzpicture}
\end{center}

Since, we have $m(v,H)<m(H)$ for some $v$, we know that at least one vertex of $H$ does not belong to any dense subgraphs of $H$. As a result, we will only need to cover a linear fraction of the vertices of \gnp\ with these dense subgraphs, since our factor will contain at least $n/\nV$ vertices to be covered by copies of these less dense vertices.

Once we have embedded the dense sub-graphs, we then use Theorem \ref{MainTheoremPartionMultigraphResult}, treating these (collapsed in $H$) embedded graphs as single vertices and finding a new, equally or less dense, \supGr -factor on these collapsed vertices, along with the remaining uncovered vertices of \gnp. This will translate to the required factor in our original graph.

To do this for the graph above, we would simply require a generalisation that allows us to partition our vertices and choose which vertex of \supGr\ will `cover' the vertices of \gnp\ in our factor. However, in a more general case, after collapsing vertices in $H$ we may no longer be left with a graph, but a multigraph, hence the required level of generalisation to use this method.

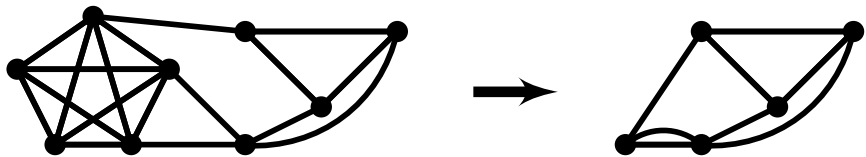
\newcommand{\BThickness}{0.75mm}
\newcommand{\BCirclePt}{4pt}
\newcommand{\BScale}{1}
%\draw[line width=\BThickness]

\begin{center}

\pgfdeclarelayer{background}
\pgfsetlayers{background,main}
\begin{tikzpicture}[scale=\BScale]

	%Creates the top vertex at {x,y} called node {} with label {}
	%\path (5,-2) node (X1) {};
	%\fill (X1.east) circle (1pt);

		\path (1,2) node (K1) {};
		\path (1.5,1) node (K2) {};
		\path (2.5,1) node (K3) {};
		\path (3,2) node (K4) {};
		\path (2,2.7) node (K5) {};

		\fill (K1.east) circle (\BCirclePt);
		\fill (K2.east) circle (\BCirclePt);
		\fill (K3.east) circle (\BCirclePt);
		\fill (K4.east) circle (\BCirclePt);
		\fill (K5.east) circle (\BCirclePt);

		\path (4,1) node (A1) {};
		\fill (A1.east) circle (\BCirclePt);

		\path (5,1.5) node (T1) {};
		\fill (T1.east) circle (\BCirclePt);

		\path (4,2.5) node (T2) {};
		\fill (T2.east) circle (\BCirclePt);

		\path (6,2.5) node (T3) {};
		\fill (T3.east) circle (\BCirclePt);

		\foreach \j in {1,2,3,4,5}
		{
			\foreach \i in {1,2,3,4,5}
			\draw[line width=\BThickness] (K\j.east)--(K\i.east);
		}

		\foreach \j in {1,2,3}
		{
			\foreach \i in {1,2,3}
			\draw[line width=\BThickness] (T\j.east)--(T\i.east);
		}

		\draw[line width=\BThickness](K3.east)--(A1.east);
		\draw[line width=\BThickness](K4.east)--(A1.east);

		\draw[line width=\BThickness](A1.east)--(T1.east);

		\draw[line width=\BThickness](K5.east)--(T2.east);

		%\draw[bend right](A1.east)--(T3.east);

		\draw[line width=\BThickness] (A1) to[bend right=40] node {} (T3.east);

		%Invisible nodes
			\path (7,1.7) node (hidden1) {};
			\path (8.4,1.7) node (hidden2) {};
			\path [draw, -latex', line width=4]  (hidden1) -- (hidden2);

		\path (9,1) node (C1) {};
		\fill (C1.east) circle (\BCirclePt);

		\path (10,1) node (B1) {};
		\fill (B1.east) circle (\BCirclePt);

		\path (11,1.5) node (P1) {};
		\fill (P1.east) circle (\BCirclePt);

		\path (10,2.5) node (P2) {};
		\fill (P2.east) circle (\BCirclePt);

		\path (12,2.5) node (P3) {};
		\fill (P3.east) circle (\BCirclePt);
		\foreach \j in {1,2,3}
		{
			\foreach \i in {1,2,3}
			\draw[line width=\BThickness] (P\j.east)--(P\i.east);
		}

		\draw[line width=\BThickness](C1.east)--(B1.east);
		\draw[line width=\BThickness](C1.east)--(B1.east);

		\draw[line width=\BThickness](B1.east)--(P1.east);

		\draw[line width=\BThickness](C1.east)--(P2.east);

		%\draw[bend right](B1.east)--(P3.east);

		\draw[line width=\BThickness] (B1) to[bend right=40] node {} (P3.east);
		\draw[line width=\BThickness] (B1.east) to[bend right=40] node {} (C1.east);

\end{tikzpicture}

\label{fig:K5Colapsing}

\end{center}

In this example, the densest subgraph is clearly the $K_5$, so we would collapse this to a single vertex. However, one vertex of $H$ contains edges to two vertices of this subgraph, leaving us with a multigraph. It is also worth noting that the density of $K_5$ is $2.5$ and since every vertex has degree at least 3, this is an example of a non-balanced graph, that could not be solved by the minimum degree result \cite{CambridgeJournals:1771816_Alon_Noga_yuster_Raphael_Factors}, and where we can provide the optimal threshold, improved on that provided by Theorem \ref{upperboundgeneral_factors}.

To define our collapsing method formally, we begin with some observations on the effects of vertex collapsing on the density of $H$. We know that $m(v,H)=m(H)$ for some vertices $v$, and these are the vertices that we collapse. For each such $v$, we, in turn, choose a subgraph $H'$ such that $v\in H'$ and $d(H')=m(H)$. We now collapse all the vertices in $H'$ into a single vertex. Giving us a new (possibly multi) graph, which we will call $H_1$, which has the vertices of $H\backslash H'$, with an additional vertex $v_1$, and an edge for each edge of $H$ with an endpoint in $H\backslash H'$. We continue this process, going from $H_i$ to $H_{i+1}$, at each stage, collapsing a subgraph of density $m(H)$ until none remain. The final graph which contains no subgraphs of density $m(H)$, we will call \supGr. We prove the rigour of this statement in the following lemma.

\begin{lemma}
The collapsing process, described above, terminates after a finite number of steps, producing a unique multigraph \supGr, with $m(\supGr )<m(H)$.
\end{lemma}

\begin{proof}
Firstly note that the density of $H_1$, defined in the same way for multigraphs as for graphs, is
\begin{equation}\label{densitySubGraph}
d(H_1)=\frac{e(H_1)}{(v(H_1)-1)}=\frac{e(H)-e(H')}{v(H)-v(H')+1-1}=\frac{e(H)-e(H')}{(v(H)-1)-(v(H')-1)}.
\end{equation}
Noting that $d(H)\leq m(H)=d(H')$, we can see that the above gives us, $d(H_1)\leq d(H)\leq m(H)$. If instead of $H$ and $H_1$, we consider any subgraph of $H$ containing the vertices we are going to collapse and the resulting subgraph of $H_1$, the same inequality shows that we have not created any subgraph in $H_1$ of density greater than $m(H)$. In fact, considering the following for positive numbers $a,b,c$ and $d$;

\[
\frac{a-c}{b-d} \geq \frac{a}{b} \Longleftrightarrow \frac{a}{b} \geq \frac{c}{d}
\]
(assuming $b>d$), and noting that we only have equality on one side if we have it on both, it follows from (\ref{densitySubGraph}) that any vertex that is in a subgraph of density $m(H)$ in $H_{i+1}$, must have also been in such a graph in $H_i$.  

Since we are considering $H$ such that, at least one vertex $v$, satisfies $m(v,H)<m(H)$, the above shows that the collapsing process will never produce a subgraph of density $m(H)$ containing these vertices and hence they will never be collapsed. This ensures that once all subgraphs have been collapsed, we will not be left with a single point, and that $m(\supGr)<m(H)$.

The above also demonstrates that while the choice of dense subgraph to collapse will result in different $H_i$, ultimately, this process will always terminate with the same final multigraph, which we call \supGr. To see why this follows, suppose a vertex lies in two different subgraphs, which we could choose to collapse. By (\ref{densitySubGraph}) applied to the subgraph induced by the union of the two dense subgraphs, the new subgraph, formed by the collapsing process, will still have density $m(H)$ and so the remaining vertices will be collapsed at a later stage to the same point.
\end{proof}

It is clear that if we can embed the collapsed, dense subgraphs of $H$, required for a factor, and then embed the edges of \supGr\, we will have our required factor. Firstly, we prove that we can embed these dense subgraphs as required. Consider the graph $H'$ with vertex set $V(H)$ and edge set $E(H)-E(\supGr)$, (i.e. $H'$ contains only those edges collapsed by the above process). Let $\th_{H'}(n)$ be the threshold for embedding a factor of $H'$ into \gnp.

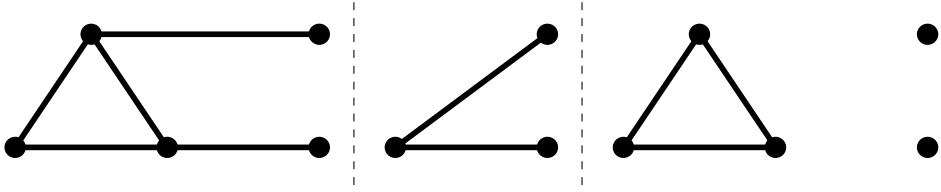
\begin{figure}[h]
\begin{center}
\pgfdeclarelayer{background}
\pgfsetlayers{background,main}
\begin{tikzpicture}[scale=\genericScale]

	%draw nodes from X20 to X41

		\foreach \j in {1,3,5}
		{
			\path (\j,1) node (X\j1) {};
			\fill (X\j1.east) circle (\genericCirclePt);
		}
		\foreach \j in {2,5}
		{
			\path (\j,2.5) node (X\j3) {};
			\fill (X\j3.east) circle (\genericCirclePt);
		}

			\draw[line width=\genericThickness ] (X11.east)--(X31.east);
			\draw[line width=\genericThickness] (X23.east)--(X31.east);
			\draw[line width=\genericThickness] (X11.east)--(X23.east);
			\draw[line width=\genericThickness] (X23.east)--(X53.east);
			\draw[line width=\genericThickness] (X51.east)--(X31.east);

			\draw [dashed] (5.6, 0.5) -- (5.6,3);
			\path (6,1) node (X61) {};
			\fill (X61.east) circle (\genericCirclePt);
			\path (8,1) node (X81) {};
			\fill (X81.east) circle (\genericCirclePt);
			\path (8,2.5) node (X82) {};
			\fill (X82.east) circle (\genericCirclePt);
			\draw[line width=\genericThickness] (X61.east)--(X81.east);
			\draw[line width=\genericThickness] (X61.east)--(X82.east);
			\draw [dashed] (8.6, 0.5) -- (8.6,3);

		\foreach \j in {9,11,13}
		{
			\path (\j,1) node (X\j1) {};
			\fill (X\j1.east) circle (\genericCirclePt);
		}
		\foreach \j in {10,13}
		{
			\path (\j,2.5) node (X\j3) {};
			\fill (X\j3.east) circle (\genericCirclePt);
		}

			\draw[line width=\genericThickness] (X91.east)--(X111.east);
			\draw[line width=\genericThickness] (X103.east)--(X111.east);
			\draw[line width=\genericThickness] (X91.east)--(X103.east);

			\typeout{end}

\end{tikzpicture}

\end{center}
\caption{An example of an $H$ and its respective \supGr\ and $H'$ graphs}
\end{figure}
\begin{lemma}\label{embedDenseSubgraphs}
For a non-vertex-balanced graph $H$, and the corresponding $H'$ as defined above, the threshold for the existence of an $H'$-factor satisfies;
 \[\th_{H'}(n)=\Theta\left(n^{-1/m(H)}\right).\]

\end{lemma}
\begin{proof}
$H'$ is a subgraph of $H$, and so $m(H')\leq m(H)$, and since $H'$ contains copies of the collapsed subgraphs of density $m(H)$, it must itself satisfy $m(H')=m(H)$. The edges of $H'$ are exactly those that were collapsed in the process that generated \supGr, and hence, any vertices of $H$ that were not collapsed, will be isolated in $H'$. In other words, these vertices will have degree $0$ in $H'$. Such vertices must exist, since we assumed that $m(v,H)<m(H)$ for some vertex $v$. Since we then have $\delta(H')<m(H')$, we can apply Theorem \ref{mindegree_factors}, completing the proof as required.

\end{proof}

We now have an $H$-factor, without the edges from each copy of $H$, that are also present in \supGr. To embed these final edges, we now use our generalisation of \ref{upperboundgeneral_factors}, namely a specific application of Theorem \ref{MainTheoremPartionMultigraphResult} to find a factor of \supGr, with the collapsed subgraphs covered by the vertices we require from $H$.

As in \cite{JKV}, we work in (a multigraph generalised form of) $G(n,M)$, the graph chosen uniformly from all $M$-edge graphs on $V$ (although we will use a multigraph generalised form of \gnp\ to prove our results) and derive a generalised form of \ref{strictbal_factors}. Since we will be operating with multigraphs and partitioned vertex sets, we need to define some notation.

Let $G$ be a graph on $n$ vertices and let $H$ be a fixed graph with vertex set $\{x_1,x_2,x_3,\dots ,x_{\nV}\}$. Let \supGr\ be the multigraph obtained by repeated applications of vertex collapsing of subgraphs of $H$ of density $m(H)$, until no such subgraphs remain. Let $\nVH=|V(\supGr)|$ and $\nEH=|E(\supGr)|$. If \supGr\ has vertex set $\{y_1, y_2, \dots ,y_{\nVH}\}$, we use lemma \ref{embedDenseSubgraphs} to find partial factors consisting of $n/\nV$ copies of these collapsed subgraphs (or single vertices, for those that were not collapsed) for each vertex in \supGr. 

We use the standard method of partitioning the edges of \gnp\ into $G(n,p')\cup G(n,p')$ where there is an edge in \gnp\ if and only if there is an edge in at least one of the $G(n,p')$. Since we are only interested in threshold functions, which are equivalent up to constant factors, and $1-p=(1-p')^2$, which implies that $ p'>p/2$, we can apply Lemma \ref{embedDenseSubgraphs} without sacrificing randomness of the edges between these embedded subgraphs. This leaves us with \nVH\ separate classes of the vertices of \gnp, each containing graphs of density $m(H)$ or isolated vertices, each corresponding to a vertex of \supGr.

We now wish to find an \supGr-factor between these partitioned sets, but we are only interested in factors that connect the `correct' vertices together from each partition, and hence are not interested in the edges within each partition, or those that are not the prescribed edges between the  subgraphs we have already embedded.
%where $n_{\nVH}=\nVH n/v$,

We can consider a random multigraph, which we call \supGNP, using the edges of our second $G(n,p')$, such that if the required \supGr -factor exists here, it will translate into the required $H$-factor in \gnp. Firstly, the vertex set of \supGNP\ consists of a single vertex for each of the isolated vertices and subgraphs of density $m(H)$, that we have embedded into \gnp. We maintain the partition of these new vertices into equal sets of size $n/\nV$, according to the vertex of \supGr, they correspond to in the initial embedding. Note that this means that \supGNP\ does not have $n$ vertices, but rather $\nVH n/\nV$, which is however, a constant multiple of $n$. 

For the edge set of \supGNP, we use the second edge partition $G(n,p')$, to ensure randomness. We consider a pair of vertices, $v_1$ and $v_2$ in \supGNP, noting that we can also consider $v_1$ and $v_2$ as sets of vertices of \gnp, and the mapping $\sigma: v(H)\rightarrow \gnp$ that describes the already embedded subgraphs that form the vertices of \supGNP. For each $x_1$ and $x_2$ $\in$ $v(H)$, with $(x_1,x_2)\in e(H)$ and $\sigma(x_1)\in v_1$ and $\sigma(x_2)\in v_2$, if $(\sigma(x_1),\sigma(x_2))\in e(\gnp)$, then we have an edge between $v_1$ and $v_2$ in \supGNP, noting that we consider each such edge separately. In this way, any factor of \supGr\ found in \supGNP\ will automatically translate into a factor of $H$ in \gnp.

\supGNP\ can also be thought of as a random $\nVH$-partite multigraph with $\nVH n/\nV$ vertices, and edges between vertices $x$ and $y$ with probability $p$ for each edge between their origin vertices in $H$, and $0$ otherwise. Essentially, a series of $\nEH=e(\supGr)$ bipartite graphs, in the same `shape' as \supGr. We will work using this random graph model (and the corresponding \supGNP\ model), to prove our results. 

It may be helpful for some readers to visualise this as a $\nVH$-partite random graph with different edge probabilities for some of the edge partitions, rather than a multigraph. For example, a single edge with probability $p^2$, rather than two edges, each with probability $p$ between vertices. The varying probabilities make this model cumbersome to work with, however, and the multigraph notation is more convenient for use. 

For some choices of $H$ it may simply be possible to set all edge probabilities to the minimum of these values, and still find the factor, but our earlier graph, containing a $K_5$, is an example of a graph for which this method would fail.

At this point, we have the exact set-up for Theorem \ref{MainTheoremPartionMultigraphResult}. As we have shown that $m(\supGr)<m(H)$, it implies that for $p=\omega(n^{-1/m(H)})>\Order(n^{-1/m(\supGr)+o(1)})$ (or, if $\supGr$ is strictly balanced $p>\omega (n^{-1/m(\supGr)}(\log n)^{1/|E(\supGr)|})$) such an \supGr -factor, a.a.s exists, and hence our $H$ factor exists in our original \gnp, as required.

\section{Theorem \ref{MainTheoremPartionMultigraphResult}}\label{Section:proofOfMultiGraphResult}

The proof of Theorem \ref{MainTheoremPartionMultigraphResult} largely follows the same steps as the original graph result in
 \cite{JKV}. To illustrate the key ideas, we outline the initial setup and then draw out several of the key ideas of the proof. We begin the proof of Theorem \ref{MainTheoremPartionMultigraphResult}, with a multigraph generalised version of their Theorem 3.1. This Theorem essentially shows that the number of factors in \supGNP\ is close to expectation, by demonstrating that the equivalent process of removing edges from the complete graph, does not remove too many factors at each step.
\begin{theorem}\label{Number of factors}
Let $\nVH=v(\supGr)$ and $\nEH=e(\supGr)$. For $p=p(n)=\omega(n^{-1/m(H)})$ and $M=M(n)=\nEH(n/\nV)^2p$, and let $\Phi(G)$ be the number of the \supFact s in a graph $G$, then
\[\Pr(\Phi(\supGNP)\geq (n^{\nVH-1}p^{\nEH})^{n/\nV} e^{-\Order(n)})\geq 1-n^{-\omega(1)}.\]

\end{theorem}
\begin{proof}
The proof follows in the same steps as the original. Let $T=\nEH(n/\nV)^2-M$, and let $e_1,e_2,\dots,e_{\nEH(n/\nV)^2}$ be a random, uniform ordering of the complete form of our multigraph (i.e \supGNP\ with $p=1$), which we shall call $KM_{n}$. Set $G_i$ to be $KM_{n}-\{e_1,e_2,\dots,e_i\}$, so if we let $\mathcal{F}(G)$ be the set of \supGr\ factors on $G$, we define $\mathcal{F}_i:=\mathcal{F}(G_i)$. We then let $\xi_i$ be the fraction of members of $\mathcal{F}_i$ containing $e_i$. Then, as for the standard graph case, we have 
\[|\mathcal{F}_t|=|\mathcal{F}_0|\frac{|\mathcal{F}_1|}{|\mathcal{F}_0|}\dots\frac{|\mathcal{F}_t|}{|\mathcal{F}_{t-1}|}=|\mathcal{F}_0|(1-\xi_1)\dots(1-\xi_t),\]
and that therefore, we have 
\begin{equation}\label{logF_Tbound}\log |\mathcal{F}_t|=\log |\mathcal{F}_0| +\sum_{i=1}^{t} \log(1-\xi_i).\end{equation}
Here our sums start to differ somewhat from the standard graph case; we have
%CHECK THIS
\[\log |\mathcal{F}_0| =\log ((n/\nV)!)^{\nVH-1}=\frac{\nVH-1}{\nV}n\log n - \Order(n).\]
Also, we have
\[\E\ \xi_i = \frac{\nEH n/\nV}{\nEH (n/\nV)^2-i+1}=:\gamma_i=\E[\xi_i|e_1,\dots,e_{i-1}]\]
for any choice of $e_1,\dots,e_{i-1}$. 

Therefore we have
\[\sum_{i=1}^t \E\xi_i=\sum_{i=1}^t \gamma_i=\frac{\nEH n}{\nV}\log \frac{\nEH (n/\nV)^2}{\nEH (n/\nV)^2-t}+o(1)\]
provided that $\nEH (n/\nV)^2-t>\omega(n)$. We use the same property as for the standard graph case, namely $\mathcal{A}_t$, which is the event that
\[\left\{ \log|\mathcal{F}_t| >\log|\mathcal{F}_0| - \sum_{i=1}^t \gamma_i -\Order(n)\right\}.\]
As before, we aim to show that with high probability $\mathcal{A}_t$ does not fail, i.e.
\begin{equation}\label{event_failure}\mbox{for } t\leq T, \Pr(\overline{\mathcal{A}_t})=n^{-\omega(1)}.\end{equation}
This implies our theorem, since we then, setting $t=T$, have
\[\log \Phi(\supGNP) = \log |\mathcal{F}_t| >\frac{\nVH-1}{\nV}n\log n + \frac{\nEH n}{\nV}\log p -\Order(n) \]
(since $M=\nEH (n/\nV)^2-T=\nEH (n/\nV)^2p$). To prove (\ref{event_failure}), we use the same methods as \cite{JKV}, namely an Azuma's inequality, martingale argument. As before, we will define two auxiliary properties $\mathcal{B}_i$ and $\mathcal{R}_i$ for $i\leq 1 \leq T-1$, that will allows us to establish control over the concentration of our variables. We set our martingale to have a difference sequence of
\[Z_i=\left\{\begin{array}{ll} \xi_i-\gamma_i \mbox{ if } \mathcal{B}_i \mbox{ and }\mathcal{R}_i \mbox{ hold for all } j<i \\
0 \mbox{ otherwise}.
 \end{array}\right.\]
 
And so our martingale is $X_t=\sum_{i=1}^{t} Z_i$. We leave the formal definitions of $\mathcal{B}_i$ and $\mathcal{R}_i$ for Section \ref{Section:PropertiesRAndB}, but in general terms, $\mathcal{R}_i$ states that each vertex is in close to expectation, number of copies of \supGr, along with a second technical condition, while $\mathcal{B}_i$ states that the maximum number of factors using a particular copy of \supGr, is close to the average over all copies of \supGr.
 For all $i\leq t$, we will have that $\mathcal{B}_{i-1}$ and $\mathcal{R}_{i-1}$ imply
 \begin{equation}\label{bound_on_xi}\xi_i=o(\log^{-1} n).\end{equation}
Our martingale analysis, will give us that $\Pr(|X_t|>n)<n^{-\omega (1)}$ (i.e. $|X_t|\leq\Order(n)$), and if we have $\mathcal{B}_i$ and $\mathcal{R}_i$ for $i< t \leq T$, we will then have that $X_t=\sum_{i=1}^{t}\xi_i-\gamma_i$ and therefore
\[\sum_{i=1}^{t}\xi_i<\sum_{i=1}^{t}\gamma_i +\Order(n) <\Order(n\log n).\]
Using this, (\ref{logF_Tbound}), (\ref{bound_on_xi}) and the series expansion for $\log(1-x)$ we get that
\[\log |\mathcal{F}_t|>\log |\mathcal{F}_0|-\sum_{i=1}^{t}(\xi_i+\xi_i^2)>\log |\mathcal{F}_0|-\sum_{i=1}^{t}\gamma_i-\Order(n).\]
As in the graph case, we are left with three possibilities for the failure of this to occur and hence,
\[\Pr(\overline{\mathcal{A}}_t)<\sum_{i<t} \Pr(\overline{\mathcal{R}}_i) +\sum_{i\leq t} \Pr(\wedge_{j<i}(\mathcal{B}_j \mathcal{R}_j) \wedge \overline{\mathcal{A}}_i)+\sum_{i<t}\Pr(\mathcal{A}_i \mathcal{R}_i \overline{\mathcal{B}}_i).\]
The previously mentioned martingale analysis shows that the second term is at most $n^{-\omega(1)}$, and we follow the same processes as \cite{JKV} in Section \ref{Section:RemainingResults} to show that for $i\leq T$
\begin{equation}\label{RdoesNotFail18}
\Pr(\overline{\mathcal{R}}_i)<n^{-\omega(1)}
\end{equation}
and
\begin{equation}\label{AandRmeansB19}
\Pr(\mathcal{A}_i \mathcal{R}_i \overline{\mathcal{B}}_i)<n^{-\omega(1)}.
\end{equation}
These three bounds give us the required result.
\end{proof}

In the next sections we outline the generalisation of the various results from the original factors paper, and include some notes on how we can apply them to our multigraph situation. Rather than just present a slightly modified reproduction of \cite{JKV}, and to make this generalisation of more value to the reader, we have first drawn out what Johansson, Kahn and Vu referred to as `the heart of the matter' and presented it as a stand-alone result, with our required generalisation and then continuing with the surrounding proofs.

As demonstrated above, the factor result follows from showing that 

\[
\Pr(\overline{\mathcal{R}}_i)<n^{-\omega(1)}
\]
and
\[
\Pr(\mathcal{A}_i \mathcal{R}_i \overline{\mathcal{B}}_i)<n^{-\omega(1)}.
\]

In proving the second equation, a second graph property $\mathcal{C}$ is introduced. The following shows that the failure of $\mathcal{C}$ results in two sets differing by a single vertex, such that the number of factors on the complement of these sets (subject to some restrictions) vary significantly. The proof of this revolves around the use of entropy results that we describe below, while in Section \ref{SectionCFailUnlikely}, concentration results are used to demonstrate that the event is unlikely as required.

\subsection{Entropy}
We follow the results of chapter 6 of the factors paper but are left with a modification to make to their Lemma 6.1. As in the original, we have $H(X)$ to be the base entropy of a discrete random variable $X$, i.e.,
\[H(X)=\sum_x p(x) \log \frac{1}{p(x)},\]
where $p(x)=\Pr(x=X).$ Now, in our case, given a vertex $y$ in a random multigraph $G$, we use $X(y, G)$ to be the copy of \supGr\ in a randomly chosen \supGr -factor, and that $\entropyFunction (y, G)=H(X(y, G))$. We will require a slightly different result, than in the original, as we will only be interested in vertices $y$ from a single partition set of our random multigraph. We suppose that our multigraph has the same structure as \supGNP, i.e. any copy of \supGr\ will contain one vertex from each partition set of $G$. Given $V_1$, a partition set of $G$, we have the following
\begin{lemma}\label{entropy}
\[\log \Phi (G) \leq \sum_{y\in V_1} h(y, G).\] 

\end{lemma}

\begin{proof}
This result follows in the same way as in the original, using a variant statement of Shearer's Lemma. This states that, given a random vector $Y=(Y_i:i\in I)$, and $\mathcal{S}$, a collection of subsets of $I$, with repeats allowed, such that each element of $I$ belongs to at least $t$ members of $\mathcal{S}$. For $S\in\mathcal{S}$ let $Y_S$ be the random vector $Y_i:i\in S$. Then $H(Y)\leq t^{-1} \sum_{S\in\mathcal{S}} H(Y_S)$. If we let $Y$ be the indicator for the random \supFact, then $I$ is the set of copies of \supGr\ in the complete form of our multigraph and $\mathcal{S}$ is the collection of sets $S_x$, where $S_x$ is the collection of copies of \supGr\ containing a vertex $x$, taken only over $x\in V_1$. We have, therefore, that each copy of \supGr, belongs to exactly one $S_x$ in $\mathcal{S}$, and so $t=1$. It follows that;
\[H(Y)=\sum_{\Phi(G )}\frac{1}{\Phi(G )} \log (\Phi(G ))=\log (\Phi(G ))\]
and since $H(Y_S)=h(y, G)$, the proof is complete.
\end{proof}

The second entropy result of \cite{JKV}, namely Lemma 6.2, is not specific to graphs, and hence requires no generalisation for our uses. We state it below for reference.

We let $S$ be a finite set, $W: S\rightarrow \mathcal{R}^+$, and let $X$ be the random variable taking values in $S$ with probability

\[\Pr(X=x)=W(x)/W(S),\]
where, for a set $A\subseteq S$, $W(A)$ is the sum of $W$ over the members of $A$, i.e. $W(A)=\sum_{x\in A} W(x)$.

\begin{lemma}\label{2ndEntropyLemma}
If $H(X) > \log |S| -\Order(1)$, then there are $a,b \in range(W)$ with 
\[a\leq b < \Order(a)\]
such that for $J=W^{-1}[a,b]$ we have,
\[|J|=\Omega(|S|)\]
and
\[W(J)>0.7W(S).\]

\end{lemma}

\subsection{The heart of the matter}\label{the_heart_of_the_matter}

We let $\Phi(G)$ be the number of \supGr -factors on a partitioned multigraph $G$,  $\mathcal{V}_0$ be the set of vertex sets of size \nVH\ in \supGNP\ with a vertex from each partition set, and $\supGr(x,G)$ be the set of copies of \supGr\ in $G$ containing the vertex $x$, again with each vertex from a separate partition set. We define $D(x,G)=|\supGr(x,G)|$ to be the number of copies of \supGr\ in $G$ containing a vertex $x$, while $D(p)$ is the expectation of $D(x,\supGNP)$ in \supGNP\ given a randomly chosen $x$.

For $Z$, a disjoint union of elements of $\mathcal{V}_0$, we define $\weight:\mathcal{V}_0 \rightarrow \mathcal{R}^{+}$ as $\weight(Z)=\Phi (\supGNP\setminus Z)$, i.e. the number of partial \supGr -factors in \supGNP\, with only the vertices in $Z$, not covered.

Lastly, fixing a set of vertices, $Y$ of size $\nVH-1$, taken from separate partition sets and two vertices $x$ and $y$, both from the remaining partition set in $V(\supGNP\setminus Y)$ we define $\wxy : \supGr(x,\supGNP-\{Y\cup \{y\}\}) \rightarrow \mathcal{R}^{+}$ as $\wxy (K)=\weight (K\cup Y \cup\{y\})$. In simple terms, $\wxy$ can be thought of as the number of $H$-factors on $\supGNP\setminus \{Y\cup\{y\}\}$ that use $K$ as the copy of \supGr\ containing $x$ in the \supGr -factor. 

\begin{definition}\label{definitionAp}
We say \supGNP\ satisfies $\mathcal{A}(p)$ if the following holds
\[\log(\Phi(\supGNP))>\frac{\nVH-1}{\nV}n\log n + \frac{\nEH n}{\nV}\log p -\Order(n).\]
\end{definition}

\begin{definition}\label{definitionOfRb}
We say that \supGNP\ satisfies $\mathcal{R}_b(p)$ if the following condition holds.

For each $x\in V$, $|D(x,\supGNP)-D(p)|=o(D(p))$.
\end{definition}
Informally, $\mathcal{A}(p)$ says that the number of factors is close to expectation, while  $\mathcal{R}_b(p)$ says the same for the number of copies of \supGr\ that each vertex of \supGNP\ is in.

For a $\nVH-1$ subset of $V(\supGNP)$, as always with each vertex taken from different partition sets; $Y$, let $\mathcal{V}_0(Y)$ be the set of $\nVH$ subsets containing $Y$, with the final vertex taken from the remaining partition set.
\begin{definition}\label{definitionOfC}
We define $\mathcal{C}$ for \supGNP\ as follows: \supGNP\ satisfies $\mathcal{C}$ if for all $\nVH-1$ subsets of \supGNP , $Y$ as above, we have the following:
\[\max \weight (\mathcal{V}_0(Y))\leq \max\{n^{-2(\nVH-1)}\Phi(\supGNP), 2\mathrm{med}  \weight (\mathcal{V}_0(Y))\}\] 
\end{definition}

We prove the following Theorem;

\begin{theorem}\label{TheoremBadSetsExist}
$\mathcal{AR}_b\mathcal{\overline{C}}$ implies that there exists a set of vertices $Y$, each taken from $|Y|=\nV-1$ different partition sets of \supGNP, and $x,y$ in the remaining partition set, such that we can find a collection $J$ of elements of $H(y,\supGNP -(Y\cup\{x\}))$ and $J'$ from $H(x,\supGNP -(Y\cup\{y\}))$ with $|J|>\Omega(|H(y,\supGNP -(Y\cup\{x\}))|)$, and $\wyx^{-1}|J|=\wxy^{-1}|J'|=[a,b]$ with $a\leq b< \Order(a)$ satisfying

\[
\sum_{X\in J}\wyx(X)>0.7\weight (Y\cup\{x\})
\]
and
\[
\sum_{X\in J'}\wxy(X)\leq 0.5\weight (Y\cup\{x\})
\]

\end{theorem}

\begin{proof}
Suppose that $\mathcal{A}$, $\mathcal{R}_b$ hold but that $\mathcal{C}$ fails. Therefore we can find at least one set $Y$ at which $\mathcal{C}$ fails. We therefore know that there exists $x$, such that $\weight (Y\cup\{x\})$ is maximum for choices of $x$ and satisfies
\[\weight (Y\cup\{x\})>n^{-2(\nVH-1)}\Phi(\supGNP).\]
We now choose $y$ with $\weight (Y\cup \{y\})\leq$ med $ \weight (\mathcal{V}_0(Y))$, and $h(y,\supGNP-(Y\cup \{x\}))$ maximal, given this constraint.

Given $\mathcal{A}$, we know that 
\[\log(\Phi(\supGNP))>\frac{\nVH-1}{\nV}n\log n + \frac{\nEH n}{\nV}\log p -\Order(n).\]
While the failure of $\mathcal{C}$ tells us that \[\weight (Y\cup\{x\})=\Phi(\supGNP-(Y\cup\{x\}))>n^{-2(\nVH-1)}\Phi(\supGNP),\]
hence, combining the two we have,
\begin{equation}\label{PhiGWithoutY}
\log \Phi (\supGNP-(Y\cup\{x\}))>\frac{\nVH-1}{\nV}n\log n + \frac{\nEH n}{\nV}\log p -\Order(n).\end{equation}

We use Lemma \ref{entropy} and apply it to the graph $\supGNP-(Y\cup\{x\})$. Letting $V_1$ be the partition set containing $x$ and $y$, we have,
\begin{equation}\label{logPhileq}
\log \Phi (\supGNP-(Y\cup\{x\}))\leq \sum_{z\in V_1} h(z, \supGNP-(Y\cup\{x\})).
\end{equation}

However, we know that we chose $y$ to have maximal entropy, chosen from a set of at least half of possible such $z$, and we also have that for any random variable $X$, the entropy $H(X)\leq \log(|\mathrm{range}(X)|)$ (with equality only if the variable is uniformly distributed).

The range of our random variable is contained in the set of copies of $H$ containing the fixed vertex $z$ in $\supGNP-(Y\cup\{x\})$, which by definition, is of size $D(z, \supGNP-(Y\cup\{x\}))\leq D(z, \supGNP)$. We know from $\mathcal{R}_b$ that this is less than $(1+o(1))D(p)$. Hence, we have that at least half the $z$'s in (\ref{logPhileq}) satisfy $h(z, \supGNP-(Y\cup\{x\}))\leq h(y, \supGNP-(Y\cup\{x\}))$ and the remaining, $n/2\nV$ all satisfy $h(z, \supGNP-(Y\cup\{x\})\leq \log( (1+o(1))D(p))$. Therefore, (\ref{logPhileq}) gives us,
\begin{eqnarray}
\lefteqn{\log(\Phi (\supGNP-(Y\cup\{x\})))\leq \sum_{z\in V_1} h(z, \supGNP-(Y\cup\{x\}))} \notag\\
&\leq &\frac{n}{2\nV}(h(y, \supGNP-(Y\cup\{x\})+\log(1+o(1))D(p))) \notag\\
&\leq &  \frac{n}{2\nV}(h(y, \supGNP-(Y\cup\{x\}) +(\nVH-1)\log n +\nEH\log p).\notag
\end{eqnarray} 

Rearranging, to get $h(y, \supGNP-(Y\cup\{x\}))$ on the left, and substituting from (\ref{PhiGWithoutY}) we have,
\[h(y, \supGNP-(Y\cup\{x\}))\geq (\nVH-1) \log n +\nEH \log p -\Order(1).\]

By $\mathcal{R}_b$ we have that 
\begin{eqnarray}
\log (D(y, \supGNP-(Y\cup\{x\})))&\leq& \log D(y,\supGNP)\leq \log ((1+o(1))D(p)) \notag\\
&=&(\nVH-1) \log n +\nEH \log p +\log (1+o(1)),\notag
\end{eqnarray} 
and hence combining with the above, we have
\begin{equation}\label{RangeOfh}
h(y, \supGNP-(Y\cup\{x\}))>\log(D(y, \supGNP-(Y\cup\{x\}))-\Order(1).
\end{equation}

We now use our functions $\wyx$ and $\wxy$, previously defined as;
\[\wyx(K)=\weight (K\cup Y \cup\{x\}) \mbox{ and similarly }\wxy(K)=\weight (K\cup Y \cup\{y\}).\]
With $\wyx$ defined on $H(y,\supGNP-(Y\cup\{x\}))$; the set of copies of $\supGr$ containing $y$ in $\supGNP-(Y\cup\{x\})$, and similarly, $\wxy$ defined on $H(x,\supGNP-(Y\cup\{y\}))$. Simply put, for a copy of $\supGr$ containing $y$ in $\supGNP-(Y\cup\{x\})$, $\wyx$ is the number of $\supGr$-factors on this set, using that copy of $\supGr$.

If we consider the random variable $X(y,\supGNP-(Y\cup\{x\}))$, which is the copy of $\supGr$ containing $y$ in a uniformly at random chosen $\supGr$-factor on $\supGNP-(Y\cup\{x\})$, we can see that the probability that $X(y,\supGNP-(Y\cup\{x\}))=\supGr'$ for $\supGr'\in \supGr(y,\supGNP-(Y\cup\{x\})$, is 

\[\wyx(\supGr')/\sum_{Z\in \supGr(y,\supGNP-(Y\cup\{x\}))}\wyx(Z).\]
Also note that the denominator is equal to $\weight (\supGNP-(Y\cup\{x\}))$, since by summing only over copies of $\supGr$, we are counting each $\supGr$-factor exactly once. 

Similarly, $X(x,\supGNP-(Y\cup\{y\}))$ is determined by $\wxy$, and the sum $\sum_{Z}\wxy(Z)$ is equal to $\weight ((Y\cup\{y\}))$.

By the above, we have the setup used for Lemma \ref{2ndEntropyLemma}, with $S=\supGr(y,\supGNP-(Y\cup\{x\}))$. Noting that $|S|=D(y, \supGNP-(Y\cup\{x\}))$, (\ref{RangeOfh}) gives us the required condition, and we are able to apply the result to $\wyx$. This implies that there exist $a$ and $b\in$ $range(\wyx)$, for which we can set $J:=\wyx^{-1}([a,b])$, and it will satisfy the following:
\[|J|>\Omega(|\supGr(y,\supGNP-(Y\cup\{x\}))|)\]
and
\[\sum_{z\in J}\wyx(Z)>0.7 \sum_{z\in H(y,G-(Y\cup\{x\})} \wyx(Z)=0.7\,\weight (Y\cup\{x\}).\]
In simple terms, $J$ is of the same magnitude in size as the whole pre-image of $\wyx$, and its elements have overall weight at least a constant multiple of that of the whole set.

Equally we can set $J'=\wxy^{-1}([a,b])$, and we know that \[\sum_{J'}\wxy(Z)\leq \weight (Y\cup\{y\})<0.5\weight(Y\cup\{x\})\]
The first inequality follows from simply summing over the full set containing $J'$, and the 2nd from our original definition of $y$ and $x$. This completes the proof.
\end{proof}

Proving that this is a.a.s. unlikely to happen, requires a range of concentration and technical lemmas, demonstrated in the following sections. In applying this result to Shamir's problem, if instead of considering factors, a matching of hyperedges is required, it has been shown that the proof follows with much more ease using a union bound argument, reducing the technical complexity of the proof considerably \cite{FRI12}.

% but it has been shown by ... TODO FRIESE %HERE that if instead of a factor of graphs, you look %at a matching for hyperedges, the result %follows much in a simpler fashion, by a union bound %argument.
%

\section{Generalisation of remaining results from \newline \cite{JKV}}\label{Section:RemainingResults}

The generalisation to partitioned structures and multigraphs of the remaining results and properties of \cite{JKV}, follow largely from careful consideration of sums and bounds, and formulation of polynomials. The following sections follow the structure of \cite{JKV} closely, and are largely a technical exercise, that offer little to those who have read the original paper. 

To highlight why the generalisation should follow, we note that while limiting the factors to these partitions appears to drastically limit the number of possible copies of $H$, since each partition is of linear in $n$ size, we still have $\Order(n^{\nVH})=\Order\binom{n}{\nVH}$, possible choices of vertices for each $H$, as in the standard case.

We also address the threshold required for applying Theorem \ref{MainTheoremPartionMultigraphResult} in obtaining Theorem \ref{MainTheoremIntro}. We are not guaranteed strict balance for the resulting \supGr, but regardless, the collapsing process, eliminates all subgraphs of density $m(H)$, and hence, $m(\supGr)<m(H)$ and so for $p>n^{-1/m(H)}$, we have a greater probability than required within the proof and so with Theorem \ref{Number of factors}, applied to \supGNP, on the partial factors already embedded during the collapsing process, we have Theorem \ref{MainTheoremIntro} as required. 

Throughout the proofs, for clarity in understanding our main result, we treat \supGr, as the graph formed by the collapsing process on some $H$, and that we have $p>n^{-1/m(H)}$ as in Theorem \ref{MainTheoremIntro}, but for proving Theorem \ref{MainTheoremPartionMultigraphResult} in full generality, \supGr\ may not necessarily be derived from some $H$, and we only have $p>n^{-1/m(\supGr)+o(1)}$ (or with a log term for the strictly balanced case). In this case the proof is unchanged, as throughout, as in \cite{JKV}, we only require that if $p=\Order(p^{-1/a})$, \supGr\ contains no subgraphs of density equal to $1/a$, and that $n^{\nVH-1}p^{\nEH}=\omega(\log n)$, which follows immediately from the conditions in Theorem \ref{MainTheoremPartionMultigraphResult}, given that the $o(1)$ term decreases sufficiently slowly.

\subsection{Concentration Results}\label{concentration}
Firstly we address the usage of the various concentration results in Section 5 of \cite{JKV}. These results are largely special cases of results by V. Vu that can be found in \cite{DBLP:journals_combinatorica_KimV00} (with J.H.Kim),   \cite{Vu:2000:CMP:349225.349231} and \cite{RSA:RSA10032_VanVu_NonLipschitz}.

We will utilise the various polynomial results here without modification, and hence will not repeat the proofs here again. 

We will require some of the notation used in this section later, which we outline now. We let $f=f(t_1,t_2,\dots,t_n)$, be a polynomial of degree $d$ with real coefficients. We say $f$ is normal if its coefficients are positive, with the maximum coefficient being $1$, and we note that the results here are also true for $\Order(1)$ normal polynomials, which simply means that the polynomial's coefficients have some fixed bound.

 We will consider polynomials that are multilinear which means we can express $f$ in the form $f(t)=\sum \alpha_U t_U,$ where $U$ ranges over subsets of $[n]$ and $t_U:=\prod_{u\in U}t_u$. 
 
 Lastly, we need that for a set $L\subseteq [n]$, the partial derivative of order $|L|$ with respect to the variables indexed by $L$ is $\sum_{U\supseteq L} \alpha_U t_{U\backslash L}$, and its expectation, denoted $\E_L$ or $\E_L f$ is $\sum\{ \alpha_U \prod_{i\in U\backslash L} p_i : U\supseteq L\}$, where $t_i\sim Ber(p_i)$. Set $\E_j f= \max_{|L|=j} \E_L f$. We write $\E_L^{'} =\E_L^{'} f$ for the expectation of the non-constant part of the partial derivative of $f$, with respect to $L$, noting that for homogeneous, $f$ of degree $d$, and $0<|L|<d$, we have $\E_L^{'} f=\E_L f$.

We take the original example used to illustrate the usage of these results, namely that we consider our polynomial $f$ to be the number of copies of \supGr\ in our random multigraph \supGNP\, containing a particular, fixed vertex $x_0$. 

We have that $f=\sum_U t_U$ where $U$ runs over edge sets of copies of \supGr\ in our complete multigraph, containing our vertex $x_0$. We have $\E f=\Theta( n^{\nVH-1}p^{\nEH})$, while for any non empty subset $L$ of edges from the complete graph, the partial derivative will be $\E_L f=\sum_{U\supseteq L}t_{U\backslash L}$. This will be 0 for $L$ that do not satisfy our multigraph structure requirements (i.e. at most one edge from each bipartite pairing forming the multigraph), in the same way that choosing an $L$ not forming a subgraph of $H$ would do, in the standard graph case. 

In all theorems in the chapter we are interested in ensuring that the maximum value of the derivative does not exceed a certain magnitude, and so in this sense, we are not interested in these cases, and so they cause no issue in this generalisation. 

Given that our choice of $L$ does satisfy our structure requirements, (and hence will be contained within at least one copy of \supGr\ in the complete graph), we can consider the graph formed by the edges of $L$, and the vertex end-points of these edges. Letting $\nVH'$ and $\nEH'$ be the number of vertices and edges respectively of $L$, then if $L$ contains $x_0$, we have $\E_L f=\Order(n^{\nVH-\nVH'}p^{\nEH-\nEH'})$, and $\Order(n^{\nVH-\nVH'-1}p^{\nEH-\nEH'})$ otherwise. Either way, we have,
\[\E f / \E_L f =\Omega (n^{\nVH'-1}p^{\nEH'}).\]
While we do not have strict balance of \supGr, we do have that it contains no subgraphs of density $m(H)$, and hence we have that $\nEH '/(\nVH'-1)<m(H)$ and recalling that $p=\omega(n^{-1/m(H)})$, we have that $\E f=\Omega(1)$ and that $\E f / \E_L f \geq n^{\Omega(1)}$, as is required for applying the results in the chapter.

For reference we include the results from the chapter below.

\begin{theorem}\label{5.1Concentration}
The following holds for any fixed positive integer $d$ and positive constant $\epsilon$. Let $f$ be a multilinear, homogeneous, normal polynomial of degree $d$ such that $\E f\geq n^\epsilon \max_{1\leq j \leq d} \E_j f$. Then
\[\Pr(|f-\E f| >\epsilon \E f)=n^{-\omega(1)}.\]
\end{theorem}

\begin{theorem}\label{5.3Concentration}
The following holds for any fixed positive integer $d$ and positive constant $\epsilon$. Let $f$ be a multilinear, normal, homogeneous polynomial of degree $d$ such that $\E f=\omega (\log n)$ and $\max_{1\leq j \leq d-1} \E_j f\leq n^\epsilon$. Then
\[\Pr(|f-\E f| >\epsilon \E f)=n^{-\omega(1)}.\]
\end{theorem}

\begin{theorem}\label{5.4Concentration}
The following holds for any fixed positive integer $d$ and positive constant $\epsilon$. Let $f$ be a multilinear, normal, homogeneous polynomial of degree $d$ such that $\E f=\omega (\log n)$ and $\max_{1\leq j \leq d-1} \E_j f\leq n^\epsilon \E f$. Then
\[\Pr(|f-\E f| >\epsilon \E f)=n^{-\omega(1)}.\]
\end{theorem}

\begin{corollary}\label{5.5Concentration}
The following holds for any fixed positive integer $d$ and positive constant $\epsilon$. Let $f$ be a multilinear, normal, homogeneous polynomial of degree $d$ such that $\E f\leq A$ where $A=A(n)$ satisfies
\[A\geq \omega (\log n) +n^\epsilon \max_{0< j < d} \E_j f,\]
then

\[\Pr(f> (1+\epsilon) A ) \leq n^{-\omega(1)}.\]
\end{corollary}

\begin{theorem}\label{5.6Concentration}
The following holds for any fixed positive integer $d$ and positive constant $\epsilon$. Let $f$ be a multilinear, normal polynomial of degree $d$ with $\E f=\omega (\log n)$ and $\max_{L\neq \emptyset} \E_L^{'} f\leq n^\epsilon \E f$. Then
\[\Pr(|f-\E f| >\epsilon \E f)\leq n^{-\omega(1)}.\]
\end{theorem}

\begin{corollary}\label{5.7Concentration}
The following holds for any fixed positive integer $d$ and positive constant $\epsilon$. Let $f$ be a multilinear, normal polynomial of degree $d$ such that $\E f\leq A$ where $A=A(n)$ satisfies
\[A\geq \omega (\log n) +n^\epsilon \max_{L\neq \emptyset} \E_L^{'} f,\]
then

\[\Pr(f> (1+\epsilon) A ) \leq n^{-\omega(1)}.\]
\end{corollary}

\begin{theorem}\label{5.8Concentration}
The following holds for any fixed positive integer $d$ and positive constant $\epsilon$. Let $f$ be a multilinear, normal polynomial of degree $d$ with $\max_{L} \E_L^{'} f\leq n^\epsilon $. Then for any $\beta (n)=\omega(1)$,
\[\Pr(f>\beta (n))= n^{-\omega(1)}.\]
\end{theorem}

\subsection{Martingale}\label{Martingale_Section}

The proof of the bound on the martingale follows exactly as that for the original paper. We have all the same bounds, namely that $|Z_i|<\varepsilon := \log^{-1} n$, and that $\sum_{i=1}^t \gamma_i = \Order (n\log n).$ and the proof makes no use of the graph setting for the problem. We include the steps below for reference.

We let $X_t=Z_1+\dots +Z_t$, and aim to show that
\[\Pr (X_t\geq n)<n^{-\omega(1)}.\]
We have that $Z_i$ is a function of the random sequence $e_1,\dots, e_i$, but that $\E(Z_i|e_1,\dots, e_{i-1})=0$ for any choice of the $e_j$'s. Using (\ref{bound_on_xi}), it will follow from the properties $\mathcal{R}$ and $\mathcal{B}$, that $|Z_i|<\varepsilon :=\log^{-1}n$. We can apply Markov's inequality to derive the following, for any positive $h$;
\begin{equation}\label{25BoundOnX>N}\Pr(X_t\geq n)=\Pr(e^{h(Z_1+\dots+Z_t)}\geq e^{hn})\leq\E(e^{h(Z_1+\dots+Z_t)})e^{-hn}.\end{equation}

Using $Z_i=\xi_i-\gamma_i$, we have that $\E(\xi_i|e_1,\dots, e_{i-1})=\gamma_i$. Using $0\leq \xi_i \leq \varepsilon$ and the convexity of $e^x$, we have
\[\E(e^{hZ_i}|e_1,\dots,e_{i-1})\leq e^{-h\gamma_i}\left(\left(1-\frac{\gamma_i}{\varepsilon}\right)+\frac{\gamma_i}{\epsilon}e^{h\varepsilon}\right).\]
Taylor series expansions, show that the right hand side is at most $e^{h^2\epsilon\gamma_i}$, for any $0\leq h\leq 1$. Using induction on $t$ we derive the following.

\[
\begin{array}{ll}
\E (e^{h(Z_1+\dots+Z_t)}) &=\E(\E (e^{h(Z_1+\dots+Z_t)}|e_1,\dots,e_{t-1}))
\\ &=\E( e^{h(Z_1+\dots+Z_{t-1})}\E(e^{hZ_t}|e_1,\dots,e_{t-1})) 
\\
&\leq\E( e^{h(Z_1+\dots+Z_{t-1})} e^{h^2\varepsilon\gamma_t})
\\
&\leq e^{h^2\varepsilon\sum_{i=1}^t \gamma_i}.
\end{array}
\]
Combined with (\ref{25BoundOnX>N}), we have
\[\Pr(X_t\geq n) \leq e^{h^2\varepsilon\sum_{i=1}^t \gamma_i -hn}.\]

We have that $\sum_{i=1}^t \gamma_i =\Order(n \log n)$ and $\varepsilon=\log^{-1} n$, and so setting $h$ to be a sufficiently small positive constant, leaves the right hand side as $e^{-\Omega(n)}=n^{-\omega(1)}$, as required.

\subsection{The Properties $\mathcal{B}$ and $\mathcal{R}$}\label{Section:PropertiesRAndB}
We now define our slightly altered properties $\mathcal{B}_i$ and $\mathcal{R}_i$. We note, as before that in proving (\ref{RdoesNotFail18}) and (\ref{AandRmeansB19}) we can operate in the random graph $\mathcal{H}(n,p_i)$ rather than $G_i$, where 
\[p_i=1-\frac{i}{\nEH (n/\nV)^2}.\]
We will define graph properties $\mathcal{B}$ and $\mathcal{R}(p)$ and then the event $\mathcal{B}_i$ will be $\{G_i$ satisfies $\mathcal{B}\}$ and $\mathcal{R}_i$ will be $\{G_i$ satisfies $\mathcal{R}(p_i)\}$.

In defining $\mathcal{B}$, we use the same notation for the functions $W$, namely that for a finite set $A$ and $W: A\rightarrow [0,\infty)$, set
\[\overline{W}(A)= |A|^{-1} \sum_{a\in A} W(a),\]
\[\max W(A) = \max_{a \in A} W(a),\]
and lastly that,
\[\mbox{maxr } W(A)=\overline{W}(A)^{-1} \max W(A),\]
with med $W(A)$ the median of $W$ on $A$.

For a multigraph $G$ with our required partition structure, and vertex set $V$, let $Z$ be a choice of $\nVH$ vertices from $V$, with each element taken from a different vertex partition. We then let $\weight_G(Z)=\Phi(G-Z)$. Therefore, $\weight_G(Z)$ is the number of \supGr -factors in the multigraph induced by $G$ on the vertex set $V\backslash Z$. 

It could also be thought of as the number of \supGr -factors in $G$, containing $Z$ as a copy of \supGr, if all edges between the vertices of $Z$ had been added in, where they are not already present. 

We also use $\weight_G(K)=\weight_G(V(K))$ for $K\in \supGr(G)$, the set of copies of \supGr\ in $G$ (which will contain one vertex from each partition of $G$).

We now define property $\mathcal{B}$, for a multigraph G as for the graph case, namely
\begin{equation}\label{PropertyBDefn}
\mathcal{B}(G)=\{\mbox{maxr } \weight_G(\supGr(G)))=\Order(1)\}.
\end{equation}
As in the graph case, $\mathcal{B}(G)$ states that no copy of \supGr\ in $G$ is contained in much more than the average number of \supGr -factors, for a copy of \supGr.

We define $\mathcal{R}(p)$, for the most part, in the same manner as for a graph, with two parts to the definition. For the first part, we use almost the same set-up. We have $G$, our random multigraph and $V$ its vertex set, and given $A\subseteq V(\supGr)$, $E'\subseteq E(\supGr)\backslash E(\supGr[A])$, an injection $\psi$, from $A$ to $V$ (mapping vertices to the correct partition of $V$ corresponding to their location in \supGr). We let $X(G)$ be the number of injections $\phi : V(\supGr) \rightarrow V$ with
\begin{equation}\label{equiv_maps}\phi \equiv \psi \mbox{ on A}\end{equation}
and
\[xy\in E' \Rightarrow \phi(x)\phi(y) \in E(G).\]

We can write $X(\supGNP)$ in an obvious way, as a polynomial in variables $t_e= {\bf{1}}_{\{e\in E(\supGNP)\}},$ for $e$, an edge in the complete form of our multigraph:
\[X(\supGNP)= q(t)=\sum_\phi t_{\phi(E')},\]
where $t$ is the indicator function for edges of the multigraph, and the sum is over all injections $\phi$ satisfying (\ref{equiv_maps}).

As in the original paper, we have that this function $q(t)$, is multilinear, $\Order(1)$-normal and homogeneous of degree $d=|E'|$. We use the same definition for $\E^*$; 
\begin{equation}\label{EStar}
\E^* = \max \{\E_L q: |L| <d\}.
\end{equation}

We also use the same definition for $D(p)$, using it as the expected  number of copies of \supGr\ in \supGNP\, using a given vertex $x\in V$, while $D(x,G)$ is the actual number of copies containing $x$ in $G$. It is clear that
\[D(p)=(n/\nV)^{\nVH-1}p^{\nEH}=\Theta(n^{\nVH-1}p^{\nEH}).\]
We are now ready to define $\mathcal{R}(p)$, which is identical to the graph formulation, not taking into account our slight changes to the above notation.
\begin{definition}\label{definitionOfR}
We say that a random multigraph $G$ satisfies $\mathcal{R}(p)$ if the following two conditions hold.
\begin{enumerate}
\item[(a)] For $A$, $E'$ and $\psi$ (and associated notation) as above: if $\E^*=n^{-\Omega(1)}$, then for any $\beta(n)=\omega(1)$,  $X(G)<\beta(n)$ for large enough $n$; if $E^* \geq n^{-o(1)}$, then for any fixed $\epsilon >0$ and large enough $n$,  $X(G) < n^\epsilon \E^*.$
\item[(b)] For each $x\in V$, $|D(x,G)-D(p)|=o(D(p))$
\end{enumerate}

\end{definition}

We can now prove that these conditions give the required bound on the size of $\xi_i$. The proof follows in the same fashion as in the original paper:

\begin{lemma}
For $i\leq T=\nEH (n/\nV)^2-M$ as defined in Section \ref{sectionOutline}, $\mathcal{B}_{i-1}$ and $\mathcal{R}_{i-1}$  (i.e. that $G_{i-1}$ satisfies $\mathcal{B}$ and $\mathcal{R}(p_i)$) imply (\ref{bound_on_xi})
\end{lemma}
\begin{proof}
Write $\weight$ for $\weight_{G_{i-1}}$. We aim to show that $\mathcal{B}_{i-1}$ and $\mathcal{R}_{i-1}$ imply that, for any $K \in \supGr(G_{i-1})$,
\[\weight(K)/\Phi(G_{i-1})=\Order(1/D(p_{i-1})).\]
As before, the left hand side of this equation is the fraction of \supGr -factors in $G_{i-1}$ that use $K$, and we prove this result in the same fashion, since we have;
\[\begin{array} {lcl} \Phi(G_{i-1}) & = & \frac{\nV}{n} \weight(\supGr(G_{i-1})) \\
& = & \frac{\nV}{n} \Omega (|\supGr(G_{i-1})| \max \weight(\supGr(G_{i-1})) \\
& = &\Omega(D(p_{i-1}) \weight(K)). \end{array}\]
The first line follows, since each \supGr -factor will be counted $n/\nV$ times by summing the $\weight$ function over all copies of \supGr. The second line follows from applying $\mathcal{B}$, while the third comes from part (b) of $\mathcal{R}_{i-1}$, and noting that $(\nVH n/\nV)D(p)=\nVH|\supGr(G_{i-1})|$.

We also have, using the same arguments as in the original, that part (a) of $\mathcal{R}_{i-1}$ implies that the number of $K \in\supGr(G_{i-1})$, containing a given edge $e\in E(G_{i-1})$ is at most  $\beta(n)$, satisfying $\beta^{-1} D(p_{i-1}) =\omega (\log n ).$ Here we also use that $D(p_{i-1})=n^{\nVH-1}p_{i-1}^{\nEH}=\omega(\log n)$, for $i\leq T$.
\end{proof}

Lastly we state the required `$p$-version' of $\mathcal{A}_t$:
\[\mathcal{A}(p)=\left\{\log |\mathcal{F}(G)| > \log |\mathcal{F}_0| - \sum_{i=1}^t \gamma_i -\Order(n)\right\},\]
where $t=\lceil(1-p)\nEH (n/\nV)^2\rceil$. Recall from the end of Theorem \ref{Number of factors}, that our required result will follow from the following lemmas
\begin{lemma}\label{RHappensLemma}
For $p>\omega(n^{-1/m(H)})$,
\[\Pr(\supGNP \mbox{ satisfies } \mathcal{R}(p)) = 1-n^{-\omega(1)}.\]
\end{lemma}
\begin{lemma}\label{AAndRThenBLemma}
For $p>\omega(n^{-1/m(H)})$,
\[\Pr(\supGNP \mbox{ satisfies } \mathcal{A}(p)\mathcal{R}(p)\overline{\mathcal{B}}) = n^{-\omega(1)}.\]
\end{lemma}
We prove these lemmas in the next subsections.
\subsection{Regularity}
We now prove Lemma \ref{RHappensLemma}, i.e. that with probability $1-n^{-\omega(1)}$, $G=\supGNP$ satisfies both parts of the definition of $\mathcal{R}(p)$.

Part (a) follows easily in the same fashion as the original paper. There are only $n^{\Order(1)}$ choice for each of $A,\ E'$ and $\psi$, so we simply show that the probability of one of these violating (a) is $n^{-\omega(1)}$. As we noted earlier, we can express the random variable as a polynomial, and since it is homogeneous, multilinear, and $\Order(1)$-normal, we can apply the probability results from the concentration chapter of \cite{JKV} directly, which gives us the result as required.

Proving (b) requires only slightly more adaptation to our scenario. We again express $D(x,G)$ as a polynomial of degree $\nEH=|E(\supGr)|$, in the variables $t_e=\bf{1}_{e\in E(G)}$, where $e$ belongs to the complete form of our random multigraph. Therefore we have,
\[D(x,G)=f(t) :=\sum \{t_K: K\in \supGr_0(x)\},\]
where $\supGr_0(x)$ is the set of copies of of \supGr, containing $x$, in the complete multigraph, and $t_K=\prod_{e\in K} t_e$.

As noted earlier, we have,
\[\E f=(n/\nV)^{\nVH-1}p^{\nEH}=\Theta(n^{\nVH-1}p^{\nEH})=\omega(\log n).\]
We aim to use one of the concentration results from \cite{JKV}, namely Theorem \ref{5.4Concentration}, which requires the above and that $\max_{1\leq j\leq d-1}\E_j f \leq n^{-\epsilon} \E f$. For $L$ a subset of the edges of the complete multigraph, with $1\leq |L|=l<\nEH$, we have
\[\E_L f = p^{\nEH-l} N(L),\]
where $N(L)$ is the number of $K\in \supGr_0(x)$ with $L\subseteq E(K)$. Let $I=V(L) \cup \{x\}$, $V(L)$ being the set of vertices incident to the edges of $L$, and $\nVH'=|I|$. Then $N(L)=\Theta(n^{\nVH-\nVH'})$ if the graph $\supGr ':=(I,L)$ is isomorphic to a subgraph of \supGr, and zero otherwise. Therefore we have, recalling that $p>n^{-1/m(H)}$ and $d(H)=\nE/(\nV-1)$,
\[\begin{array}{lcl} \E f /\E_L f & = & \Omega (n^{\nVH'-1}p^l)=\Omega (n^{[(\nVH'-1)/l-1/m(H)]l}) \\
& = & \Omega (n^{[1/d(\supGr')-1/m(H)]l})=n^{\Omega(1)}
\end{array}.\]
Using the fact that \supGr\ contains no subgraphs of density $m(H)$ (or denser). We therefore have the required conditions to use the concentration theorem and the result follows, which provides us with part (b) of $\mathcal{R}$ as required.

\subsection{Proof of Lemma \ref{AAndRThenBLemma}}
We now begin the proof of Lemma \ref{AAndRThenBLemma}, continuing in the same vein as the original paper, as there, we will prove that $\mathcal{B}$ is satisfied, using an auxiliary event, $\mathcal{C}$. Most of these results follow in an identical manner to the original, but with small conditions on the choice of sets, and differing constant powers in the equations (largely from use of $\nVH$ rather than $\nV$). We include these modified results for completeness.

We write $\mathcal{V}_0$ to be the collection of $\nVH$-sets of $V=V(\supGNP)$, with a single vertex from each partition of \supGNP. For a set $Y\subseteq V$, with $|Y|\leq\ \nVH$ and at most one vertex from each partition, we write $\mathcal{V}_0(Y)$ for the set $\{Z\in \mathcal{V}_0: Z \supseteq Y\}$. We then extend our earlier weight function $\weight=\weight_{\supGNP}$ to these sets $Y$, by setting
\[\weight(Y)=\sum\{\weight(Z): Z\in \mathcal{V}_0(Y)\}.\]
We define our new property $\mathcal{C}$ for \supGNP\ as follows: \supGNP\ satisfies $\mathcal{C}$ if for all such $Y$, as defined above, with $|Y|=\nVH-1$, 
\[\max \weight(\mathcal{V}_0(Y))\leq \max\{n^{-2(\nVH-1)}\Phi(\supGNP), 2\mbox{med } \weight(\mathcal{V}_0(Y)\}\]

We will then prove Lemma \ref{AAndRThenBLemma}, by proving the two following results.

\begin{lemma}\label{AAndRThenCLemma10.1}
$\Pr(\mathcal{AR\overline{C}})=n^{-\omega (1)}.$
\end{lemma}

\begin{lemma}\label{RCThenBLemma10.2}
$\Pr(\mathcal{RC\overline{B}})=n^{-\omega (1)}.$
\end{lemma}

We have already demonstrated the first part of the proof of Lemma \ref{AAndRThenCLemma10.1} in Section \ref{the_heart_of_the_matter}. For clarity, to follow the arguments of the original paper, we will first prove Lemma \ref{RCThenBLemma10.2}, before returning to Lemma \ref{AAndRThenCLemma10.1} in the next Section. We must firstly show that
\begin{equation}|\{K\in \mathcal{V}_0: \weight(K)\geq \delta \max \weight(\mathcal{V}_0)\}|=\Omega(|\mathcal{V}_0|)\end{equation}
Noting that $|\mathcal{V}_0|=\Omega(n^{\nVH})$, and that this implies $\max r \weight(\mathcal{V}_0)=\Order(1)$. We now need another small modification of a lemma from the graph case. We let $\psi(X)=\max \weight(\mathcal{V}_0(X)$ and let $B$ be a positive number, recall that $V$ is the vertex set of \supGNP, and is of size $\nVH n/\nV$, with $\nVH$ partitions of size $n/\nV$.

\begin{lemma}
Suppose that for each $Y\subseteq V$, satisfying $|Y|=\nVH-1$ and with at most one vertex from each partition of $V$, and $\psi(Y)\geq B$ we have
\[\left|\left\{Z\in \mathcal{V}_0(Y): \weight(Z)\geq \frac{1}{2}\psi(Y)\right\}\right|\geq \frac{n/\nV -\nVH}{2}.\]
Then for any $X \subseteq V$ with $|X|=\nVH-i$, at most one vertex from each partition and $\psi(X)\geq 2^{i-1} B$, we have
\begin{equation}\label{consequenceOfLemma36}
\left|\left\{Z\in \mathcal{V}_0(X): \weight(X)\geq \frac{1}{2^i}\psi(X)\right\}\right|\geq \left(\frac{n/\nV -\nVH}{2}\right)^i \frac{1}{(i-1)!}.\end{equation}
\end{lemma}

\begin{proof}
We write $N_i$ for the right hand side of the above equation, and proceed by induction on $i$. The case $i=1$ is the hypothesis of the lemma. We then assume $X$ as stated, and choose $Z\in \mathcal{V}_0(X)$ with $\weight(Z)=\psi(X)$ (i.e.~$Z$ such that $\weight$ is maximal). We let $y\in Z\backslash X$ and $Y=X\cup \{y\}$. We then have that $|Y|=\nVH-(i-1)$ and that $\psi(Y)=\psi(X)\geq 2^{i-1} B (\geq 2^{i-2}B)$, and so, by the inductive hypothesis there are at least $N_{i-1}$ sets $Z' \in \mathcal{V}_0(Y)$ with $\weight(Z')\geq 2^{-(i-1)}\psi(Y)$. For each such $Z',\ Z'\backslash\{y\}$ is a $(\nVH-1)$-subset of $V$ with $\psi(Z'\backslash\{y\})\geq \weight(Z')\geq B$. So then for each such $Z'$, there are at least $(n/\nV -\nVH)/2$ sets $Z''\in \mathcal{V}_0(Z'\backslash\{y\})$ with
\[\weight(Z'')\geq \psi(Z'\backslash \{y\})/2\geq 2^{-i}\psi(X).\]
Therefore the number of such pairs, $(Z,Z'')$ is at least $N_{i-1}(n/\nV-\nVH)/2$. Equally, for each $Z''$, each corresponding $Z'$ is $Z'' \backslash \{u\} \cup \{y\}$ for some $u\in Z''\backslash (X\cup\{y\})$. Therefore the number of such $Z'$ is at most $i-1$, providing our factorial term. This completes the proof.

\end{proof}

We now continue the proof of Lemma \ref{RCThenBLemma10.2}. We set $\delta = 2^{\nVH}$ and then $\mathcal{C}$ implies the hypothesis of the above lemma, with $B=(2n)^{-(\nVH-1)}\Phi(\supGNP) (>n^{-2(\nVH-1)}\Phi(\supGNP))$. We also clearly have that 
\[\psi(\emptyset)\geq n^{-(\nVH-1)}\Phi(\supGNP)=2^{\nVH-1}B.\]

We also set $\gamma = (2^{\nVH+1}(\nVH-1)!)^{-1}$, and using (\ref{consequenceOfLemma36}), we now have
\[|\{K\in\mathcal{V}_0:\weight(K)\geq \delta \max \weight(\mathcal{V}_0)\}| >\gamma n^{\nVH}.\]
We let $J$ be the largest power of 2, not exceeding $\max \weight$ and
\[\mathcal{Z}=\{Z\in \mathcal{V}_0: \weight(Z)>\delta J\}.\]
In a sense, $\mathcal{Z}$ can be thought of as the vertex sets of size $\nVH$, whose complement has a relatively large number of factors. For any set $X\subseteq V$ with $|X|\leq \nVH$, with at most a single vertex from each partition, let $\mathcal{Z}(X)=\{Z\in\mathcal{Z}: X\subset Z\}$, and say such a set $X$ is $good$ if $|\mathcal{Z}(X)|>\gamma n^{\nVH-|X|}.$ In particular we know that the empty set is $good$.
We then fix an ordering $a_1,\dots, a_{\nVH}$ of $V(\supGr)$. For distinct vertices, $x_1,\dots ,x_r \in \supGNP$, we define $S(x_1,\dots, x_r)$ to be the collection of copies $\phi$ of $\supGr$ in the complete multigraph $KM_n$, for which
\[\phi(a_i)=x_i \mbox{ for } i\in [r],\]
\[\phi(a_i)\phi(a_j)\in E(\supGNP)\mbox{ whenever } i,j\geq r \mbox{  and } a_i a_j\in E(\supGr)\]
and that $\phi(V(\supGr))\in\mathcal{Z}$.

For each $r\in \{0,\dots,\nVH\}$ let $N_r=N(a_r)\cap\{a_{r+1},\dots, a_{\nVH}\},$ and $d_r=|N_r|$, where $N(a_r)$, means the neighbourhood of $a_r$ (in \supGr). We now let $\mathcal{Y}(x_1,\dots,x_r)$ be the event
\[\{|S(x_1,\dots,x_r)|=\Omega(p^{d_r+\dots +d_{\nVH-1}}n^{\nVH-r})\}.\]
Note that in particular we have, $d_{\nVH}=0$, and, recalling that $\supGr(G)$, is the set of copies of \supGr\ in \supGNP,
\[S(\emptyset)=\{\phi\in\supGr(G): \weight(\phi(V(\supGr)))>\delta J\}\]
and that $\mathcal{Y}(\emptyset)$ is the event
\[\{|S(\emptyset)|=\Omega(p^{\nEH} n^{\nVH})\}.\]
Then, for $v_1,\dots ,x_r\in V$, and from distinct partitions, let $Q(x_1,\dots,x_r)$ be the event
\[\{\{x_1,\dots,x_r\} \mbox{ is good}\} \wedge \overline{\mathcal{Y}}(x_1,\dots,x_r)\}.\]
Since we have shown that $\mathcal{C}$ implies that the empty set is good, we have that $\overline{\mathcal{B}}\mathcal{C}\subseteq Q(\emptyset)$ and therefore, we simply require to show that 
\[\Pr(\mathcal{RQ}(\emptyset))=n^{-\omega(1)},\]
to prove Lemma \ref{RCThenBLemma10.2}, we continue (as always, in the same fashion as for the graph case), by proving a slightly more general argument for induction purposes, namely, that for any choice of $r$ and vertices $x_1,\dots,x_r,$
\begin{equation}\label{RandQisUnlikely38}
\Pr(\mathcal{RQ}(x_1,\dots,x_r))=n^{-\omega(1)}.
\end{equation}

Our induction is on $\nVH-r$, with our initial step $r=\nVH$, trivially following since the definition of being good for subsets of size $\nVH$ is to belong to $\mathcal{Z}$. For general $r<\nVH$, we set $X=\{x_1,\dots,x_r\}.$ and then let $\mathcal{P}$ be the event
\[\{y\in V\backslash X, X\cup\{y\} \mbox{ good } \Rightarrow \mathcal{Y}(x_1,\dots,x_r,y)\}.\]
By the inductive hypothesis we know that $\Pr(\mathcal{R\overline{P}})=n^{-\omega(1)}$, so we only need to show
\[\Pr(\mathcal{RP}Q(x_1,\dots,x_r))<n^{-\omega(1)}.\]
We also note that if $X$ is good then,
\begin{equation}\label{LotsOfGood40}
|y:X\cup\{y\} \mbox{ good}|=\Omega(n).
\end{equation}
To ensure that the edges between $x_r$ and $V\backslash X$ are independent of the initial conditioning, we use a relaxed form of $\mathcal{R}$, $\mathcal{R}_X$, which we say is satisfied if it satisfies part (a) of $\mathcal{R}$, whenever $A=\{a_1,\dots,a_r\},\ \psi (a_i)=x_i\ (i\in [r])$ and $E'\subseteq(H-A).$

As for the graph case, if $\mathcal{RP}\wedge\{X$ good$\}$ holds, but $\mathcal{Y}(x_1,\dots,x_r)$ does not, then there must be some $J=2^{\nVH}$, with $\nVH$ and integer not exceeding $n\log n$ (the magnitude of the $\log$ of the number of $\supFact$s in the complete multigraph), such that with $\mathcal{Z}$ good, we have the following, (noting that throughout this chapter, wherever we choose sets of vertices from $V-X$, we choose them from partitions that do not contain vertices of $X$),
\begin{enumerate}
\item[(a)] $\mathcal{R}_X$ holds;
\item[(b)] There are at least $\Omega(n)$ $y$'s in $V\backslash X$ for which we have $\mathcal{Y}(x_1,\dots,x_r,y)$ (by (\ref{LotsOfGood40})), and lastly,
\item[(c)] $\mathcal{Y}(x_1,\dots,x_r)$ does not hold.
\end{enumerate}

We note, that for a given $J$, the first two properties, depend only on $G':=G-X$. Since the number of possibilities for $J$ is at most $n\log n$, it is enough to show that for any $J$ and $G'$ satisfying (a) and (b) (with respect to $J$),
\[\Pr(\overline{\mathcal{Y}}(x_1,\dots,x_r)|G')=n^{-\omega(1)}.\]
Given this fixed $G'$, we can express $|S(x_1,\dots,x_r)|$ as a multilinear polynomial in terms of the indicator variable for edges between $x_r$ and other vertices of \supGNP, i.e.
\[t_u:={\bf{1}}_{\{x_r u\in E(\supGNP\}}\ u\in V\backslash X.\]
giving the polynomial,
\[|S(x_1,\dots,x_r)|=g(t):=\sum_U\alpha_U t_U,\]
where $U$ ranges over $d_r$ subsets of $V\backslash X$, with the vertices taken from the correct partitions, corresponding to the edges from $a_i=\phi^{-1}(x_i)$ in \supGr, and $\alpha_U$ is the number of copies $\psi$ of $K:=\supGr-\{a_1,\dots,a_r\}$ in $G'$ with the induced subgraph,
\[\psi(N_r)=U\]
and
\[\psi(\{a_{r+1},\dots,a_{\nVH}\})\cup X \in \mathcal{Z}.\]
We now apply a concentration result from Section \ref{concentration}, namely Theorem \ref{5.4Concentration}. To apply this, we require a normal polynomial, so we normalise, and consider
\[f(t)=\alpha^{-1}g(t),\]
where $\alpha$ is the maximum of the $\alpha_U$'s. The hypothesis requires that $\E f=\omega(\log n)$ and $\max_{1\leq j\leq d-1} \E_j f \leq n^{-\epsilon} \E f$, and will allow us to say that it is close to expectation. We rewrite:
\[g(t)=\sum_{y\in V\backslash X} \sum \{t_{\phi(N_r)}:\phi \in S(x_1,\dots,x_r,y)\}.\]
Since we know that the indicator variables, $t_u,\ u\in V\backslash X$, are independent of $G'$, which determines our sets $S(x_1,\dots,x_r,y)$, and using property (b) our situation, we have
\begin{equation}\label{43ExpectationOfG}
\E g = p^{d_r} \sum_{y\in V\backslash X} |S(x_1,\dots,x_r, y)|=\Omega(p^{d_r+\dots+d_{\nVH-1}}n^{\nVH-r}).
\end{equation}
Noting that if $d_r=0$, then there are no random edges to consider, and we have that $|S(x_1,\dots,x_r)|$ will equal 
\[\sum_{y\in V\backslash X} |S(x_1,\dots, x_r, y)|=\Omega(p^{d_r+\dots+d_{\nVH-1}} n^{\nVH-r}),\]
 as required, we will now assume that $d_r>0$.

We now set $\supGr'=\supGr-\{a_1,\dots, a_{r-1}\}$, and so $d_r+\dots+d_{\nVH-1}=e(\supGr')$ and that $\nVH-r=v(\supGr ')-1$. This gives us that the right hand side of the expectation of $g$, is $\Omega(p^{e(\supGr')}n^{v(\supGr')-1}).$ Using that $p>n^{-1/m(H)}$ and that \supGr\ contains no subgraphs of density $m(H)$, we have
\[ 
\E g= \left\{\begin{array}{ll} \omega(\log n) & \mbox{ if } r=1  \\
n^{\Omega(1)} & \mbox{ if } r>1.  
 \end{array}\right.
\]
We will now show that for the normalised polynomial $f(t)$, that
\begin{equation}\label{ExpectationF45}
\E f =\omega (\log n)
\end{equation}
and
\begin{equation}\label{maxExpectation46}
\max \{\E_T f: T\subseteq V \backslash X, 0<|T|<d_r\}=n^{-\Omega(1)}\E f.
\end{equation}
With these conditions, the concentration theorem will tell us that $f$, and therefore $g$ is close to its expectation, which implies $\mathcal{\overline{Y}}$, as required.
To prove the two conditions, we will find it easier to consider the partial derivatives of $g$ rather than $f$. We use $t_e=\textbf{1}_{e\in E(\supGNP)},\ t_S=\prod_{e\in S} t_e$ and $t=(t_e: e\in E(KM_{V\backslash X}))$, where $KM_{V\backslash X}$ is the multigraph induced by the `complete' multigraph $KM_n$ on the the vertex set $V\backslash X$.

Since we are only interested in establishing upper bounds on the partial derivatives of $g$, we may now disregard the second requirement on $\alpha_U$, namely $\psi(\{a_{r+1},\dots,a_{\nVH}\})\cup X \in \mathcal{Z}.$ Therefore we are left with
\[p^{-(d_r-l)}\E_T g \leq \sumOverPhiPolynomialNotH(t) :=\sum_{\phi} t_{\phi(E(K))},\]
where we sum over $\phi$, injections such that

\[\phi:V(K)\rightarrow V\backslash X \mbox{ with } \phi(N_r)\supseteq T.\]

We set $E^*$, as before to be $E^*=\max \{\E_L \sumOverPhiPolynomialNotH : L\subseteq E(KM_{V\backslash X}), |L| < |E(K)|\}.$ We will show that there is a positive constant $\varepsilon$ (depending only on \supGr), such that (for large enough $n$),
\begin{equation}\label{pLessThanEpsilon50}
p^{d_r-l} \E^* < n^{-\epsilon} \E g.
\end{equation}

This will give us the two requirements for our concentration theorem, as follows. To prove (\ref{ExpectationF45}), we need to show that
\[\alpha^{-1} \E g=\omega(\log n).\]
We apply (\ref{pLessThanEpsilon50}) with $T=U$, a $d_r$-subset of $V/X$. We consider the two possible conditions of $\mathcal{R}_X$, which gives two separate cases, firstly if $\E g \geq n^{\varepsilon/2},$ (i.e. $\E^*\geq n^{-o(1)})$ then $\mathcal{R}_X$ tells us that
\[\alpha_U=\E_U g<n^{\varepsilon/4} \max \{1,\E^*\} \leq n^{-\varepsilon/4} \E f.\]
In the other case we are left with $\E^* <n^{-\varepsilon/2}$, and hence $\E^* = n^{-\Omega(1)}$. We know that $\E g = \omega(\log n)$, and hence we choose our $\beta(n)=\omega(1)$ (from $\mathcal{R}$) such that $\beta(n)^{-1} \E g=\omega (\log n)$. Since this does not depend on our choice of $U$, we have (\ref{ExpectationF45}) as required.

To prove the other requirement, we need $\E_T g = n^{-\Omega(1)} \E g$ for any $T$, as we defined in (\ref{maxExpectation46}). We again apply (\ref{pLessThanEpsilon50}), noting that we can decrease the $\varepsilon$ without violating the equation, and use this observation to assume that $\varepsilon < 1/m(H)$. This gives us
\[\E_T g < p^{d_r -l} n^{\varepsilon/2} \max\{1, \E^*\}\leq n^{-\varepsilon/2} \E g\]
as we required.

We now return to prove (\ref{pLessThanEpsilon50}). We fix $L\subseteq E(KM_{V\backslash X})$ and let $h_l=|E(K)|-|L|,$ (recalling that $ K=\supGr-\{a_1,\dots,a_r\}).$ We know that $\E_L \sumOverPhiPolynomialNotH=p^{h_l}N_L$, where again $N_L$ is the number of $\phi$, as defined just before (\ref{pLessThanEpsilon50}), with $\phi(E(K))\supseteq L.$ Each $\phi$ satisfies $\phi(V(K))\supseteq I :=T\cup V(L)$, where as earlier, $V(L) \supseteq V\backslash X$ is the set of vertices incident to the edges of $L$. We let $I=\{i_1, \dots,i_s\}$, then we have $N_L=\sum N_L(b_1,\dots,b_s)$, where $(b_1\dots,b_s)$ range over $s$-tuples of distinct elements of $V(K)$ and we sum over the number of $\phi$'s as above, with $\phi(b_j)=i_j$ for each $1\leq j \leq s$. We only have $\Order(1)$ choices for the $b_j$'s and hence the result will follow if we can show that for any such choice,
\[p^{d_r-l+h_l} N_L(b_1,\dots,b_s)=n^{-\Omega(1)} \E g.\]
Given a fixed choice of $b_i$'s, let $\supGr ''=\supGr[\{a_r,b_1,\dots,b_s\}]$. We know that $N_L(b_1,\dots,b_s)<n^{\nVH-r-s}=n^{\nVH-r-(v(\supGr'')-1)}$, and that $h_l\geq l+d_{r+1}+\dots+d_{\nVH-1}-e(\supGr'')$, since $|E(K)|=d_{r+1}+\dots +d_{\nVH-1}$ and $E(\supGr'')$ contains $\phi^{-1}(L)$ and at least $l$ edges joining $a_r$ to $V(K)$. Therefore we have
\[p^{d_r-l-h_l}N_L(b_1,\dots,b_s)<p^{d_r+\dots +d_{\nVH-1}}n^{\nVH-r}[n^{v(\supGr'')-1}p^{e(H'')}]^{-1}.\]
Noting that since $p>n^{-1/m(H)}$ and that \supGr\ contains no subgraphs of density $m(H)$, we have that the expression in the square brackets is $n^{\Omega(1)}$. Combined with our earlier bound on $\E g$, from (\ref{43ExpectationOfG}), we have the bound required above and hence (\ref{pLessThanEpsilon50}), completing the Lemma.

\subsection{Proof of Lemma \ref{AAndRThenCLemma10.1}}\label{SectionCFailUnlikely}
We now begin the final step of the proof, namely that $\Pr(\mathcal{AR\overline{C}})=n^{-\omega (1)}$. We maintain our use of notation from the previous chapter with $G=\supGNP$, and $\mathcal{R}$ and $\mathcal{A}$. In Section \ref{the_heart_of_the_matter}, we have shown that $\mathcal{AR\overline{C}}$ results in the following;

There exist $x$ and $y$; vertices in the same partition set of $G=\supGNP$ and a set $Y$ with a single vertex from each of the remaining partitions, such that there exist, $a,b\in range(\weight_y) (=range(\weight'))$, such that $J:=\weight^{-1}_y ([a,b])$ satisfies
\[|J|>\Omega(|\supGr(y,G-R)),\]
and
\[\weight_y(J)> 0.7\weight_y(\supGr(y,G-R))=0.7\weight (R).\]
With, $J'=\weight^{-1}_x ([a,b])$, we also have,
\[\weight_y(J)>0.7 \weight(R),\]
and
\[\weight_x(J')\leq \weight(S)<0.5\weight(R).\]

All that remains is to show that the probability of this event is $n^{-\omega(1)}$.

Once again, we can express $\weight_y(J)$ and $\weight_x(J')$ as evaluations of a multi-linear polynomial in variables $\{t_u:u\in W\}$, once we have conditioned on the value of $G[W]$, in the following way. Given a set $U\subseteq W$ with $|U|\leq \nVH-1$ and at most one vertex from each partition, let $G_U^*$ be the graph obtained from $G[W]$ by adjoining a vertex $w^*$ say, with neighbourhood $U$. We let $\mathcal{K}_U$ be the set of copies of \supGr\ in $G_U^*$, containing $\{w^*: u\in U\}$, and
\[\alpha_U=\sum\{\weight'(V(K)\backslash\{w^*\}):K\in \mathcal{K}_U, \weight'(V(K)\backslash \{w^*\})\in [a,b]\}.\]
Our polynomial now becomes,
\[g(t)=\sum_{U\subseteq W}\alpha_U t_U,\]
and $\weight_y(J)$ and $\weight_x(J')$ are simply $g$ evaluated at the point $t':=\textbf{1}_{\{z\in W:yz\in E(G)\}}$ and $t'':=\textbf{1}_{\{z\in W:xz\in E(G)\}}$.

$\mathcal{R}$ tells us that $D_G(y)=\Theta(n^{\nVH-1}p^{\nEH})$ and therefore, considering that at most $o(n^{\nVH-1}p^{\nEH})$ copies of $H$ lie in $G$ and contain $y$ and meet $R$, we have $|\supGr(y,G-R)|=\Theta(n^{\nVH-1}p^{\nEH})=\omega(\log n)$. 

We know from our conditioning of $J$, that it satisfies
\[|J|=\Theta(n^{\nVH-1}p^{\nEH}),\]
and
\begin{equation}\label{58W_YBound}
\weight_y(J)=\Theta(bn^{\nVH-1}p^{\nEH}).
\end{equation}
We now apply a concentration result, namely Corollary \ref{5.7Concentration}. To apply it, we require a bound on the expectation of $f$, namely, that if $Ef\leq A$, where
\[A\geq \omega(\log n) + n^\epsilon \max_{l\neq \emptyset}E'_L f.\]
then the corollary gives us $\Pr(f>(1+\epsilon)A)=n^{-\omega(1)}.$
We fix $T\subseteq W$, with $|T|=l<\nVH$, and as always, each vertex from a different partition set. For $d=l,\dots,\nVH-1,$, and $t$ as $t=(t_e: e\in E(KM_{W}))$, we consider the polynomial,
\[h_d(t)=\sum_z\sum_{\phi} t_{\phi(E(\supGr -z))},\]
Where $z$ ranges over vertices of \supGr\ of degree $d$ and $\phi$ over the injections $V(H)\backslash \{z\}\rightarrow W$ with $\phi(N_z)\supseteq T.$ Then we know that
\[\alpha_T\leq b\ h_l(t).\]
Since $h_l(t)$ is the number of such sets, and we know that their values are bounded by $b$, this also gives us
\[\E_T ' g \leq b \sum_{d>l} p^{d-l} h_d(t),\]
where $\E_T '$ is the non-constant part of the partial derivative.

We let
\[\E_d^*=\max\{E_L h_d:L\subseteq E(KM_{W}),|L|<\nEH-d\}.\]
As in the previous chapter, similarly to (\ref{pLessThanEpsilon50}) we can assert that there is positive constant $\epsilon$ dependent only on \supGr, such that, for each $d$,
\[p^{d-l} \E_d^*<n^{-\epsilon}n^{\nVH-1}p^{\nEH}.\]
This follows from the proof of (\ref{pLessThanEpsilon50}) in the previous chapter, by noting that in our definition of $h_d$, there are only finitely many $z$, and the inner sum, is bounded by the polynomial \sumOverPhiPolynomialNotH, used in (\ref{pLessThanEpsilon50}), with $r=1$, $a_1=z$ and hence $d_r=d$.

We are now able to apply the concentration results we require. We consider the polynomial $f=\alpha^{-1} g$, where $\alpha=\max_U \alpha_U$. Then we have, using (\ref{58W_YBound}),
\[f(t')=\alpha^{-1}\weight_y(J)=\Theta(\alpha^{-1} bn^{\nVH-1} p^{\nEH})\]
and using the same arguments as we used at the end of the previous chapter, we can show
\[f(t')=\omega(\log n)\]
and
\[\max\{\E_T' f: T\subseteq W, T\neq \emptyset\}=n^{-\Omega(1)}f(t').\]
For the final step, we first summarise that we have shown $\mathcal{AR\overline{C}}$ implies that there exist $Y,\ x,\ y$ and $a,\ b\in range(\weight')$ for which we have the above two conditions, and
\[f(t'')<0.8f(t').\]
However, we know that for any given choice of $Y,\ x,\ y,\ a,$ and $b$, $f$ depends only on $G[W]$. Equally, if $G[W]$ is fixed, then $t$ and $t'$ are independent random variables, with `law' $Bin(W,p)$ (i.e. $t=(t_w: w\in W)$ has `law $Bin(W,p)$' if each $t_w$ is an independent Bernoulli with mean $p$.) We finish with the final claim of \cite{JKV}, which can be applied directly to our result here, without generalisation, and completes the proof.

\newtheorem{claim}{Claim}
\begin{claim}
For any $\epsilon >0$ and $d$ the following holds. If $f$ is a multilinear, normal polynomial of degree at most $d$ in $n$ variables, $\zeta (n) =\omega (\log n)$, and $t', t''$ are independent , each with $law\ Bin([n], p),$ then
\[\Pr(f(t')> \max \{\zeta (n), n^\epsilon \max_{T\neq \emptyset} E_{T}^{'} f , (1+\epsilon ) f(t'')\})=n^{-\epsilon}.\]
Because there are only polynomially many possibilities for $Y,x, y, a$ and $b$, this gives us Lemma \ref{AAndRThenCLemma10.1}.
\end{claim}

$Proof\ of\ claim$. Set
\[ A=\frac{1}{2} \max \left\{ \zeta(n), n^{\epsilon} \max_{T\neq \emptyset} E^{'}_{T} f.\right\}.
\]
If $\E f \leq A$ then Corollary \ref{5.7Concentration} gives
\[\Pr(f(t')> A)=n^{-\omega(1)};\]
otherwise, by Theorem \ref{5.6Concentration},
\[Pr (f(t')>(1+\epsilon) f(t'')) <\Pr(\max\{|f(t')-\E f|, |f(t'')-\E f|\} > (\epsilon /3) \E f)=n^{-\omega(1)}.\]

\qed

\section{Theorems \ref{MainTheoremPartionResult} and \ref{MainTheoremDigraphResult}}

To prove Theorem \ref{MainTheoremPartionResult}, we note that our proof of Theorem \ref{MainTheoremPartionMultigraphResult} does not require the collapsed graph to be a multigraph. Throughout, we only require that the graph does not contain any subgraphs of density $m(H)$, and this follows, in our main result, from the vertex collapsing technique, but as in \cite{JKV}, it can also follow from strict balance of \supGr, or equivalently $H$. Since our partitions were fixed only by the partial embedding of subgraphs of density $m(H)$, in the strictly balanced case, we can simply choose our vertex partitions freely, and then continue with the proof.

Lastly, to prove \ref{MainTheoremDigraphResult}, we use Theorem \ref{MainTheoremPartionResult}. Partition the vertices of \gnp\ as usual into \nV\ equal sets, either as a result of vertex collapsing dense subgraphs, or freely in the strictly balanced case. In the former case, embedding the partial factors of dense directed subgraphs, requires a simple modification of Theorem \ref{partial_factors}, which follows the same arguments as the original but for directed graphs. A proof of which will appear in \cite{AJM_PhDThesis}. We again consider the edges between the partitions, only where they correspond to edges in $H$. Now however, we can consider only the edges that are in the direction we require, which are distributed uniformly and independently at random, but with edge probability $p/2$, since we discount those edges in the wrong direction. At this point, the random digraph is equivalent to $\supGr(n,p/2)$, and we apply \ref{MainTheoremPartionResult} directly, providing the directed factor as required.

\section{Conclusion}
With these results, Conjecture~\ref{MainConjecture} is now proven for both strictly balanced graphs, and all non-vertex balanced graphs. The methods used here also can be used to prove that the conjecture holds for a wide range of vertex-balanced $H$.

We can prove a threshold for a variety of `necklace' and related graphs formed of copies of a dense graph linked by a `supergraph' of equal or lower density.

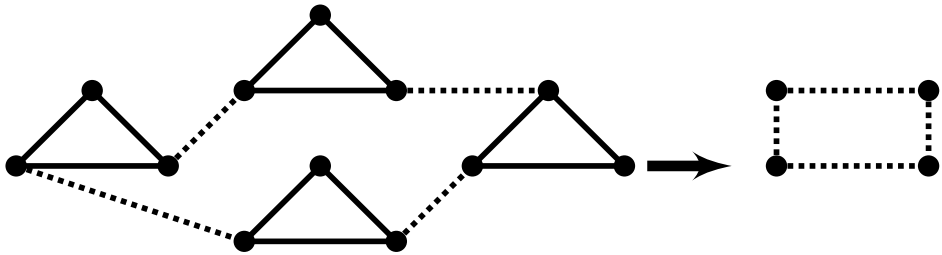
\newcommand{\AThickness}{0.75mm}
\newcommand{\ACirclePt}{4pt}
\newcommand{\AScale}{1}

\begin{center}
\pgfdeclarelayer{background}
\pgfsetlayers{background,main}
\begin{tikzpicture}[scale=\AScale]

%Red Square
		\path (11,2) node (Y12) {};
		\fill (Y12.east) circle (\ACirclePt);
		\path (13,2) node (Y32) {};
		\fill (Y32.east) circle (\ACirclePt);
		\path (11,3) node (Y13) {};
		\fill (Y13.east) circle (\ACirclePt);
		\path (13,3) node (Y33) {};
		\fill (Y33.east) circle (\ACirclePt);

	%draw nodes from X20 to X41

		\path (5,4) node (X54) {};
		\fill (X54.east) circle (\ACirclePt);

		\foreach \j in {2,4,6,8}
		{
			\path (\j,3) node (X\j3) {};
			\fill (X\j3.east) circle (\ACirclePt);
		}

		\foreach \j in {1,3,5,7,9}
		{
			\path (\j,2) node (X\j2) {};
			\fill (X\j2.east) circle (\ACirclePt);
		}

		\foreach \j in {4,6}
		{
			\path (\j,1) node (X\j1) {};
			\fill (X\j1.east) circle (\ACirclePt);
		}
		\begin{pgfonlayer}{background}
			%triangle 1
			\draw[line width=\AThickness] (X12.east)--(X32.east);
			\draw[line width=\AThickness] (X23.east)--(X32.east);
			\draw[line width=\AThickness] (X12.east)--(X23.east);
			%triangle 2
			\draw[line width=\AThickness] (X41.east)--(X61.east);
			\draw[line width=\AThickness] (X41.east)--(X52.east);
			\draw[line width=\AThickness] (X52.east)--(X61.east);
			%triangle 3
			\draw[line width=\AThickness] (X72.east)--(X92.east);
			\draw[line width=\AThickness] (X83.east)--(X92.east);
			\draw[line width=\AThickness] (X72.east)--(X83.east);
			%triangle 4
			\draw[line width=\AThickness] (X43.east)--(X63.east);
			\draw[line width=\AThickness] (X63.east)--(X54.east);
			\draw[line width=\AThickness] (X54.east)--(X43.east);

			%Invisible nodes
			\path (9.3,2) node (hidden1) {};
			\path (10.7,2) node (hidden2) {};
			\path [draw, -latex', line width=4]  (hidden1) -- (hidden2);

			\draw[line width=\AThickness][dotted] (X12.east)--(X41.east);
			\draw[line width=\AThickness][dotted] (X63.east)--(X83.east);
			\draw[line width=\AThickness][dotted] (X32.east)--(X43.east);
			\draw[line width=\AThickness][dotted] (X61.east)--(X72.east);

			\draw[line width=\AThickness][dotted] (Y12.east)--(Y32.east);
			\draw[line width=\AThickness][dotted] (Y13.east)--(Y33.east);
			\draw[line width=\AThickness][dotted] (Y12.east)--(Y13.east);
			\draw[line width=\AThickness][dotted] (Y33.east)--(Y32.east);

		\end{pgfonlayer}
\end{tikzpicture}
\end{center}

As an example, in the above case, the threshold for finding a copy of this `triangle necklace' must be at least that of a triangle factor, and Conjecture~\ref{MainConjecture} suggests it should be equal. We prove this by first embedding a triangle factor, and then use \ref{MainTheoremPartionResult} to embed the dotted edges that form a less dense cycle in the supergraph. Provided the resulting graph or multigraph from the collapsing process is less dense than $m(H)$, we will always be able to apply Theorem \ref{MainTheoremPartionMultigraphResult} to find this. 

The supergraph cannot ever be denser than $m(H)$, as this implies a subgraph of density greater than $m(H)$, which is a contradiction. So, all that remains to consider are graphs where we will be left with a supergraph of equal density, after collapsing all subgraphs. We know this must be strictly balanced, or we would have continued to collapse its subgraphs and so we can apply Theorem \ref{MainTheoremPartionMultigraphResult}. However, the supergraph may have fewer edges than each of the collapsed subgraphs, and certainly has fewer than the original graph. This may force it to require a higher threshold, by a constant power of $\log$ to embed, and so will not prove the conjecture.
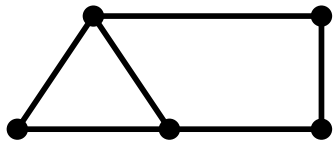
\newcommand{\QThickness}{0.75mm}
\newcommand{\QCirclePt}{4pt}
\newcommand{\QScale}{1}

\begin{center}
\pgfdeclarelayer{background}
\pgfsetlayers{background,main}
\begin{tikzpicture}[scale=\QScale]

	%draw nodes from X20 to X41

		\foreach \j in {1,3,5}
		{
			\path (\j,1) node (X\j1) {};
			\fill (X\j1.east) circle (\ACirclePt);
		}
		\foreach \j in {2,5}
		{
			\path (\j,2.5) node (X\j3) {};
			\fill (X\j3.east) circle (\ACirclePt);
		}

			\draw[line width=\AThickness] (X11.east)--(X31.east);
			\draw[line width=\AThickness] (X23.east)--(X31.east);
			\draw[line width=\AThickness] (X11.east)--(X23.east);
			\draw[line width=\AThickness] (X23.east)--(X53.east);
			\draw[line width=\AThickness] (X51.east)--(X31.east);
			\draw[line width=\AThickness] (X51.east)--(X53.east);

\end{tikzpicture}
\end{center}
For the above graph, the conjectured threshold is $n^{-2/3}(\log n)^{1/5}$, but using the methods within this paper, we are only able to embed it at $p=\Order(n^{-2/3}(\log n)^{1/3})$, since collapsing the triangle will leave us with another triangle, requiring the higher $\log$ term to embed.

In light of this, we see that we can prove Conjecture~\ref{MainConjecture} for all $H$, except for those mentioned above, for which we are still within a constant power of $\log$ of the conjectured bound.

% This is the Janson, Luczak and Rucinski book \cite{0471175412_Janson_Luczak_Rucinski_RandomGraphs}
% This is the factors paper \cite{JKV}
% This is the spanning tree paper by Krivelevich \cite{2010arXiv1007.2326K_Spanning_trees_Krivelevich}
%This is the min degree less than density paper \cite{CambridgeJournals:1771816_Alon_Noga_yuster_Raphael_Factors}

%this is the extension threshold paper \cite{DBLP:journals/jct/Spencer90a:Extension_Thresholds}

\bibliography{AJM_bibliography}

\end{document}